\documentclass[11pt,oneside,a4paper,reqno]{amsart}
\usepackage[utf8]{inputenc}
\RequirePackage{fix-cm} 
\usepackage[T1]{fontenc}
\usepackage[english]{babel}

\usepackage{lmodern}
\usepackage{amsmath}
\usepackage{amsthm}
\usepackage{amssymb}
\usepackage{bbold}
\usepackage{mathrsfs}
\usepackage{graphicx}
\usepackage{color}
\usepackage{longtable}
\usepackage{mathtools}
\usepackage{leftidx}
\usepackage{eurosym}
\usepackage{enumerate}
\usepackage{dsfont}
\usepackage{hyperref}	
\usepackage{float}
\usepackage[dvipsnames]{xcolor}
\usepackage{todonotes}

\usepackage{hyperref}
\hypersetup{
colorlinks,
citecolor=Mahogany,
linkcolor=Mahogany,
urlcolor=Mahogany}

\makeatletter
\def\@seccntformat#1{\hspace*{0mm}%
 \protect\textup{\protect\@secnumfont
   \ifnum\pdfstrcmp{subsection}{#1}=0 \bfseries\fi
   \csname the#1\endcsname
   \protect\@secnumpunct
     }%
}
 
\usepackage{geometry}
\newcommand{\N}{\mathbb{N}}
\newcommand{\Z}{\mathbb{Z}}

\newcommand{\R}{\mathbb{R}}

\newcommand{\dl}{\mathrm{d}}

\newcommand{\sph}{\mathbb{S}}

\newcommand{\sequence}[2]{(#1_{#2})_{#2}}

\newcommand{\norm}[2]{\|#1\|_{#2}}
\newcommand{\mv}[2]{\langle#1 \rangle_{#2}}

\newcommand{\veps}{\varepsilon}
\newcommand{\Om}{\Omega}

\newcommand{\Omzeps}{{\Omega^0_\varepsilon}}
\newcommand{\Omoeps}{{\Omega^1_\varepsilon}}

\newcommand{\weakly}{\rightharpoonup}
\newcommand{\wtwos}{\overset{2s}{\rightharpoonup}}
\newcommand{\stwos}{\overset{2s}{\to}}

\newcommand{\ol}[1]{\overline{#1}}

\newcommand{\per}{\text{per}}

\newcommand{\sfT}{\mathsf{T}}

\DeclareMathOperator{\curl}{curl}
\DeclareMathOperator{\diver}{div}

\newtheorem{Theorem}{Theorem}[section]

\newtheorem{Corollary}[Theorem]{Corollary}
\newtheorem{Proposition}[Theorem]{Proposition}
\newtheorem{Lemma}[Theorem]{Lemma}
\newtheorem*{Conjecture}{Conjecture}

\theoremstyle{definition}
\newtheorem{Definition}[Theorem]{Definition}
\newtheorem{Remark}[Theorem]{Remark}
\newtheorem{Setting}[Theorem]{Convention}
\newtheorem{Assumption}[Theorem]{Assumption}

\newenvironment{theorem}[1][]
{ \begin{Theorem}[#1]}
	{\end{Theorem}}

\newenvironment{definition}[1][]
{ \begin{Definition}[#1]}
	{\end{Definition} }

\newenvironment{corollary}[1][]
{ \begin{Corollary}[#1]}
	{\end{Corollary} }

\newenvironment{proposition}[1][]
{ \begin{Proposition}[#1]}
	{\end{Proposition}}

\newenvironment{lemma}[1][]
{ \begin{Lemma}[#1]}
	{\end{Lemma}}

\newenvironment{remark}[1][]
{ \begin{Remark}[#1]}
	{\end{Remark}}

\numberwithin{equation}{section}

\title[]{Two-scale density of almost smooth functions in sphere-valued Sobolev spaces: A high-contrast extension of the Bethuel-Zheng theory}
\date{\today}

\author[E. Davoli] {Elisa Davoli} 
\author[L. Happ]{Leon Happ}

\date{}
\AtEndDocument{	
	\hrule
	\medskip
	\small{
		\noindent
		{\sc E. Davoli},
		Institute of Analysis and Scientific Computing,
		TU Wien,
		Wiedner Hauptstra\ss e 8-10, 1040 Vienna, Austria.
		
		\noindent
		E-mail:
		\href{mailto:elisa.davoli@tuwien.ac.at}{\tt elisa.davoli@tuwien.ac.at}.
	}
   
    \medskip \medskip \medskip \hrule  \medskip
   
    \small{
		\noindent
		{\sc L. Happ},
		Institute of Analysis and Scientific Computing,
		TU Wien,
		Wiedner Hauptstra\ss e 8-10, 1040 Vienna, Austria.

           \centerline{\sc and}

            \noindent MedUni Wien, Währinger Gürtel 18-20, 1090 Vienna, Austria.
		
		\noindent
		E-mail:
		\href{mailto:leon.happ@student.tuwien.ac.at}{\tt leon.happ@student.tuwien.ac.at}.
	}
	\medskip
\hrule
}

\begin{document}
\begin{abstract}
		In this paper we prove a strong two-scale approximation result for sphere-valued maps in \(L^2(\Om;W^{1,2}_0(Q_0;\sph^2))\), where \(\Om\subset \R^3\) is an open domain and \(Q_0\subset Q\) an open subset of the unit cube \(Q=(0,1)^3\). The proof relies on a generalization of the seminal argument by F. Bethuel and X.M. Zheng to the two-scale setting. We then present an application to a variational problem in  high-contrast micromagnetics.
		
		\medskip
		

		
	\end{abstract}
\maketitle


\section{Introduction}
Approximating manifold-valued Sobolev mappings by smooth functions is a classical question. One of the seminal results in this direction is the theory established by F. Bethuel and X.M. Zheng in \cite{BeZh88}. On the one hand, in the setting of $W^{1,2}(B;\mathbb{S}^2)$ (where $B$ denotes the unit ball in $\mathbb{R}^3$ centered in the origin), smooth approximations in the strong $W^{1,2}$-topology are not possible (cf. Theorem~2 in \cite{BeZh88} and Section~4 in \cite{SchUh83}; see also Theorem~2.3 and Theorem~2.6 in \cite{Ha09}). On the other hand, the authors prove in \cite{BeZh88} that each map in $W^{1,2}(B;\mathbb{S}^2)$ can be approximated by almost smooth maps, namely maps which are smooth in $\mathbb{R}^3$ up to a finite number of points. 

The first contribution of this paper is to investigate whether this latter density result carries over to the setting in which the functions under consideration exhibit both a macroscopic and a microscopic scale, and the strong $W^{1,2}$-topology is replaced by strong two-scale convergence, cf. Section \ref{sec:prelims}. 

To be precise, denoting by $\Omega$ a bounded domain in $\mathbb{R}^3$ and by $Q_0$ an open subset of the unit cube $Q=(0,1)^3$, we consider maps in $L^2(\Omega;W^{1,2}(Q;\mathbb{R}^3))$ exhibiting a second scale depencence localized in $Q_0$. Our main result is the following.
\begin{theorem}\label{thm:strong_ts_approx}
Let $\tilde{m}\in W^{1,2}(\Omega;\sph^2)$ and $w\in L^2(\Omega;W^{1,2}_0(Q_0;\R^3))$ be such that 
\begin{equation*}
    \tilde{m}(x)+w(x,z)\in \sph^2 
    \quad \text{for a.e. }
    (x,z)\in \Omega\times Q_0.
\end{equation*}
Then, there exists a sequence of functions \(\sequence{m}{k}\subset W^{1,2}(\Om;\sph^2)\) such that
    \begin{alignat}{2}
        &m_k 
        \stwos \tilde{m}(x) + w(x,z)\chi^0(z)
        && \quad \text{in } L^2(\Om;\sph^2),\label{eq:first_stwos_approx}\\
        &\veps_k \nabla m_k
        \stwos \nabla_z w(x,z)\chi^0(z)
        && \quad \text{in } L^2(\Om;\R^{3\times 3}).\label{eq:second_stwos_approx}
    \end{alignat}
    Moreover, there holds
    \begin{equation*}
        m_k\in C^\infty(\Om\setminus \{a_{k_1},\dots, a_{k_l}\};\sph^2),
    \end{equation*}
    i.e., the functions \(m_k\) are smooth up to a finite number of points, where the points themselves and the number \(k_l\) depend on the index \(k\in \N\).
\end{theorem}
In comparison with \cite[Theorem~4]{BeZh88}, the constructed approximants converge in a weaker norm but \eqref{eq:second_stwos_approx} provides a finer characterization of the limiting oscillations. The proof methodology stems from a combination of the Sard's lemma application and the projection analysis highlighted in \cite[Theorem~4]{BeZh88} with two-scale convergence and unfolding. Relying on these latter tools, we actually prove a slightly refined version of Theorem \ref{thm:strong_ts_approx} (cf. Theorem \ref{thm:recov_mag}), which serves as starting point of the high-contrast analysis presented in the second part of the paper.

The mathematical modeling of high-contrast composites has been attracting increasing interest over the last couple of years. From a variational point of view, this question can be phrased as identifying the \(\Gamma\)-limit in the sense of \cite{DM93} of energy functionals modeling highly heterogeneous composites, and exhibiting a drastical difference (high contrast) in the material properties of their constituents. This leads, in particular, to a loss of coercivity of the respective energy functionals, so that limiting nonlinear theories are not always available while working with classical weak or strong Sobolev convergence, but often require turning to two-scale adaptations instead (\cite{ChCh12}, \cite{DaKrPa21}, \cite{DaGaPa2}). 

We consider here a high-contrast toy-problem motivated by the theory of micromagnetics (cf. \cite{HuSch08}, \cite{Br63}). In order to describe our results we need to introduce some notation. Let \(\Om\subset \R^3\) be a bounded domain representing a ferromagnetic body. 
Let \(E\subset \R^d\) be an open, connected, and periodic set with Lipschitz boundary. The assumption that \(E\) is periodic in this setting means that
\begin{equation*}
	E+e_i = E
	\quad \text{ for all } i=1,2,3,
\end{equation*}  
with \((e_i)_{i=1,2,3}\) the canonical basis of \(\R^3\). We then define \(Q_1\coloneqq Q\cap E\) and \(Q_0\coloneqq Q\setminus  \overline{Q_1}\). Let further 
\begin{equation}\label{eq:setup1}
	Z_\veps\coloneqq \{z\in \Z^3 : \veps (\overline{Q_0}+z)\subset \Om\}
\end{equation} 
encode those lattice points in \(\Z^3\) such that the sets \(\veps (\overline{Q_0}+z)\) are well-separated from the boundary of \(\Om\). We finally define the two constituents of the composite on the scale \(\veps>0\) as 
\begin{equation}\label{eq:setup3}
\Omzeps
\coloneqq \bigcup_{z\in Z_\veps}  \veps(Q_0+z).
\end{equation}
and
\begin{equation}\label{eq:setup2}
 \Omoeps
\coloneqq \Om\setminus \ol{\Omzeps}.
\end{equation}
This construction amounts to cutting out from some homogeneous sample \(\Om\) holes constituted by all those \(\veps (Q_0+z)\) that are compactly contained in \(\Om\), and then filling them again with a different material than the one that makes up \(\Omoeps\).

The micromagnetic behavior of $\Omega$ is encoded by the functionals
\begin{equation}\label{eq:min_probs}
	\mathcal{G}_\veps(u,m)
	:= \frac{1}{2}\int_{\Omoeps} |\nabla m |^2 \,{\dl x} 
	+\frac{1}{2}\int_{\Omzeps} \veps^2|\nabla m |^2 \, {\dl x}
	+ \int_{\R^3}|h_d[m]|^2\,{\dl x},
	\end{equation}
where \(m\), describing the magnetization of the body, is assumed to satisfy a saturation constraint and belong to the class of functions
\begin{equation*}
W^{1,2}(\Om;\sph^2)
\coloneqq \{m\in W^{1,2}(\Om;\R^3):\ m(x)\in \sph^2 \text{ for a.e. } x\in \Om\}.
\end{equation*}
The first two terms in \eqref{eq:min_probs} describe symmetric exchanges and favour constant magnetizations (with very different strength) in the two portion of the specimen. The third energy contribution, instead, has a nonlocal character, encodying the effects of the stray field. The demagnetizing field \(h_d[m]\) is defined as the image of the linear and continuous operator 
\begin{equation}\label{eq:demag}
    h_d:L^2(\R^3;\R^3)\to L^2(\R^3;\R^3),
\end{equation}
which is uniquely determined (cf. \cite{AlDF15}, \cite{SR07}) through the magneto-static Maxwell equations
\begin{equation}\label{eq:maxwell}
    \begin{cases}
        -\diver h_d[m]= \diver m\chi_\Om,\\
        \nabla \times h_d[m] =0
    \end{cases}		
    \quad \text{in } \R^3.
\end{equation}
In other words, the pair \((m\chi_\Om,h_d[m])\) fulfils
\begin{equation}\label{eq:diff_constraints}
    \mathcal{A}(m\chi_\Om,h_d[m])
    \coloneqq
    \left(\begin{array}{cc}
       \diver  &  \diver\\
        0  &     \curl
    \end{array}\right)
    \left(\begin{array}{cc}
         m\chi_\Om  \\
         h_d[m]
    \end{array}\right)
    =0
\end{equation}
in the distributional sense, i.e., it lies in the kernel of the associated constant rank differential operator \(\mathcal{A}\).
The demagnetizing fields satisfies \(\norm{h_d[m]}{L^2(\R^3)}\le \norm{m}{L^2(\Om)}\le |\Om|\) (the Lebesgue measure of \(\Om\)) for all \(m\in W^{1,2}(\Om;\sph^2)\), and the magneto-static energy in \eqref{eq:min_probs} is equivalently expressed as
\begin{equation*}
    \int_{\R^3}|h_d[m]|^2\,{\dl x}
    =-\int_{\Om}h_d[m]\cdot m\,{\dl x}.
\end{equation*}
We neglect here anisotropy contributions, external fields, as well as antisymmetric exchanges, for they would not modify substantially the analysis, acting as lower order perturbations.

A first difficulty in the identification of the high-contrast limit of $\mathcal{G}_{\veps}$ is due to the presence of the saturation constraint. As already observed in \cite{Pi04}, in the absence of Sobolev coercivity for the admissible magnetizations, a relaxation of the saturation constraint occurs in the energy both for the magnetizations and for the admissible stray fields. 
The non-local nature of the stray-field energy further deeply affects the properties of the \(\Gamma\)-limit of \(\sequence{\mathcal{G}}{\veps}\) in comparison to former results in high-contrast homogenization. In particular, in the setting analyzed here, the energy cannot be split into two separate contributions, depending only from the inclusions, or only from the matrix portion of the composite, respectively. It rather exhibits, in the limiting stray-field behavior, nontrivial coupling effects. As a result, the influence of the $\veps$-inclusions cannot be averaged out in the homogenized energy density: The coupling of the two material components through the non-local energy contribution thus calls for a tracking of the fast variable in the choice of the underlying convergence. 

\begin{definition}[Convergence in the sense of extensions]\label{def:convergence_single}
	Let \((\veps_k)_k\) be an infinitesimal sequence. We say that \(\sequence{f}{k}\subset W^{1,2}(\Om;\R^3)\) converges to \(f\in W^{1,2}(\Om;\R^3)\) \textit{in the sense of extensions}, with respect to the scales \((\veps_k)_k\), if
	\begin{enumerate}[(i)]
		\item
		\(\sequence{f}{k}\) is bounded in \(L^2(\Om;\R^3)\) and
		\item
		there exists a sequence \(\sequence{\tilde{f}}{k}\subset  W^{1,2}(\Om;\R^3)\) such that \(\tilde{f}_k=f_k\) in \(\Om_k^1\) and 
        \begin{equation}
            \tilde{f_k}\weakly f
            \qquad \text{in } W^{1,2}(\Om;\R^3).
        \end{equation}
	\end{enumerate}
\end{definition}

\begin{definition}[High contrast convergence]\label{def:conv_for_gamma}
	Let \(\sequence{\veps}{k}\) be an infinitesimal sequence. Then \((m_k)_k\subset W^{1,2}(\Om;\R^3)\) converges to the tuple \((\tilde{m},w)\) for 
	\begin{equation}
		\tilde{m}\in W^{1,2}(\Om;\R^3)
		\quad\text{and}\quad
            w\in L^2(\Om;W_0^{1,2}(Q_0;\R^3))
	\end{equation}
	 \textit{in the high contrast sense} if 
	\begin{enumerate}[(i)]
		\item
			\(\sequence{m}{k}\) converges to \(\tilde{m}\) {in the sense of extensions},
		\item
			  and
			 \begin{equation}\label{eq:ts_in_hc}
			 	m_k\wtwos \tilde{m}+w
			 	\qquad \text{ in } L^2(\Om;\R^3).
			 \end{equation}
	\end{enumerate}
 We say that \((m_k)_k\subset W^{1,2}(\Om;\R^3)\) converges strongly to \((\tilde{m},w)\) in the high-contrast sense if \(m_k\stwos \tilde{m}+w\) in \eqref{eq:ts_in_hc}.
\end{definition}

Denoting by $Q$ the unit cube, in order to formulate our \(\Gamma\)-limit result we define the mapping (see Lemma~\ref{lem:demag_limit})
    \begin{equation}
        h_d^z: L^2(\Om\times Q)\to L^2(\Om;L^2_\per(Q))
    \end{equation}
    through the identity \(h_d^z[m]\coloneqq \nabla_z r_m\), where \(r_m\in L^2(\Om;W^{1,2}_\per(Q)/\R)\) and where for almost every \(x\in \Om\) the scalar function \(r_m(x,\cdot)\) is the unique solution in \(W^{1,2}_\per(Q)/\R\) of the cell problem
    \begin{equation}
        \int_Q \nabla_z r_m(x,z)\cdot \nabla_z \psi(z)\,{\dl z}
        = - \int_Q m(x,z)\cdot \nabla_z \psi(z)\,{\dl z}
        \quad \forall \psi\in W^{1,2}_\per(Q).
    \end{equation} 
    In particular, there holds
    \begin{equation}\label{eq:diff_constraints_z}
    \mathcal{A}_z(m, h_d^z[m])
    \coloneqq
    \left(\begin{array}{cc}
       \diver_z  &  \diver_z\\
        0  &     \curl_z
    \end{array}\right)
    \left(\begin{array}{cc}
         m \\
         h_d^z[m]
    \end{array}\right)
    =0
    \qquad \text{for a.e. } x\in \Om
\end{equation}
in the distributional sense. We now define the functional
    \begin{equation}\label{eq:gamma-limit}
        \mathcal{G}(\tilde{m},w)
        \coloneqq \mathcal{F}^0(w) 
        + \mathcal{F}^1(\tilde{m}) 
        + \int_{\R^3} \mathcal{W}_{\hom}(\tilde{m},w)\,{\dl x}
    \end{equation}
    with the contributions
    \begin{align}
    	\begin{split}\label{eq:soft_gamma_lim}
    		&\mathcal{F}^0(w)
    		:=
    		\frac{1}{2}\int_{\Om}\int _{Q_0} |\nabla_z w|^2\,{\dl z}{\dl x},
    	\end{split}\\
    	\begin{split}
    	&
    	\mathcal{F}^1(\tilde{m})
    	:=\int_{\Om}\tilde{f}_{\hom}(\tilde{m},\nabla \tilde{m})\ {\dl x},
    	\end{split}
    \end{align}
    and the homogenized self-energy density
    \begin{align}\label{eq:stray_field_limit}
    \mathcal{W}_{\hom}(\tilde{m},w)
    = |h_d[(\tilde{m}+\mv{w(x,\cdot)}{Q_0})]|^2\,{\dl x} 
    + \int_{Q_0} |h_d^z[w(x,z)]|^2\,{\dl z}\chi_\Om.
    \end{align}
    The tangentially homogenized energy density \(\tilde{f}_{\hom}\), cf. \cite{BaMi10}, satisfies the cell formula
    \begin{equation}\label{eq:g_hom_special}
    \tilde{f}_{\hom}(s,\xi)
    = 
    \inf_{\varphi \in W^{1,2}_{\per}(Q;T_s\sph^2)}
    \frac{1}{2}\int_{Q_1} |\xi+\nabla \varphi(z)|^2\, {\dl z}.
    \end{equation}
    It is called "tangentially homogenized" because we consider the infimum over all functions in \(W^{1,2}_{\per}(Q;T_s\sph^2)\) with values in the tangent space \(T_s\sph^2\) of the sphere. We prove the following characterization.

\begin{theorem}\label{thm:main_gamma_conv}
    Let \(\mathcal{G}:W^{1,2}(\Om;\R^3)\times L^2(\Om;W_0^{1,2}(Q_0;\R^3))\to \R\cup \{+\infty\}\) be as in \eqref{eq:gamma-limit}. Then:
    \begin{enumerate}[(i)]
        \item 
            For any \((\tilde{m},w)\in W^{1,2}(\Om;\sph^2)\times L^2(\Om;W_0^{1,2}(Q_0;\R^3))\) and for any sequence \(\sequence{m}{\veps}\subset W^{1,2}(\Om;\sph^2)\) that converges to \((\tilde{m},w)\) in the strong high-contrast sense, there holds
            \begin{equation*}
                \mathcal{G}(\tilde{m},w)
                \le \liminf_{\veps \to 0} \mathcal{G}_\veps(m_\veps).
            \end{equation*}
        \item 
            For any \((\tilde{m},w)\in W^{1,2}(\Om;\sph^2)\times L^2(\Om;W_0^{1,2}(Q_0;\R^3))\) there exists a sequence \(\sequence{m}{\veps}\subset W^{1,2}(\Om;\sph^2)\) that converges to \((\tilde{m},w)\) in the strong high-contrast sense and that satisfies
            \begin{equation*}
                \limsup_{\veps \to 0} \mathcal{G}_\veps(m_\veps)
                \le \mathcal{G}(\tilde{m},w).
            \end{equation*}
    \end{enumerate}
\end{theorem}
We stress that the infimum problem in \eqref{eq:g_hom_special} admits a unique minimizer. Define \(\varphi\in W^{1,2}_\per(Q;\R^3)\) as the unique solution to the Poisson equation on the unit cell (where we write \(W^{-1,2}_\per(Q;\R^3)\) for the dual space of \(W^{1,2}_\per(Q;\R^3)\))
    \begin{equation*}
        -\Delta \varphi = \chi_{Q_1}
        \quad \text{in } W^{-1,2}_\per(Q;\R^3).
    \end{equation*}
     One can show, using the Lax-Milgram theorem (cf. \cite[Proposition~2.1]{DaDF20}) that if \((\tau_1(s), \tau_2(s))\) constitute an orthonormal basis of \(T_s\sph^2\), the unique minimizer of \eqref{eq:g_hom_special} can be written as
     \begin{equation*}
         \phi[s,\xi]
         \coloneqq [\varphi\cdot(\xi\tau_1(s))] \tau_1(s)
         + [\varphi\cdot(\xi\tau_2(s))] \tau_2(s).
     \end{equation*}
    The corresponding minimal energy is given by
    \begin{equation*}
        \tilde{f}_{\hom}(s,\xi)
        =\frac{1}{2}|Q_1||\xi|^2-\int_{Q_1} \nabla_z \phi[s,\xi](z)\,{\dl z}.
    \end{equation*}

 We believe that the proof of Theorem~\ref{thm:main_gamma_conv} and especially the construction of the corresponding recovery sequence provide interesting insight into the hurdles involved in nonlocal high-contrast problems in manifold-valued Sobolev spaces. 
 As already mentioned, the entanglement of the magnetization on \(\Omzeps\) and \(\Omoeps\) prohibits us to reduce the study of the full energies in Theorem~\ref{thm:main_gamma_conv}  to the analysis of the isolated contributions from these two different subsets. In fact, we can still follows this path for the difference
 \begin{equation*}
    \mathcal{F}_\veps(u,m)
    := \mathcal{G}_\veps(u,m) 
    - \int_{\R^3}\mathcal{W}(m)\,{\dl x},
\end{equation*}
where $\int_{\R^3}\mathcal{W}(m)\,{\dl x}:=\int_{\R^3}|h_d[m]|^2\,{\dl x}$,
but additional work is required when putting the full energy back together again. Especially the (approximate) recovery sequence has to account for the different effects of the local and non-local parts of the energy. This is where we rely on the result of Theorem~\ref{thm:strong_ts_approx} (to be precise, we will use a slight generalization that we provide in Theorem~\ref{thm:recov_mag}). The main advantage in separating from the energies \(\sequence{\mathcal{F}}{\veps}\) the exchange contributions on \(\Omoeps\) is given by the fact that this allows to trace back its limit behavior to a \(\Gamma\)-limit in the sense of classical homogenization (see Proposition~\ref{thm:stiff_hom}). In our case this is made possible through a special extension operator developed in \cite{GaHaPa23} that preserves the norm constraint. We recall this result in Section~\ref{subsec:extensions}.

We conclude this introduction with a final comment about the notion of convergence in Theorem \ref{thm:main_gamma_conv}. The identification of strong two-scale $\Gamma$-limits in the high-contrast framework dates back to \cite{ChCh12}.
In general, in our setting, we merely know that the functionals \(\sequence{\mathcal{G}}{\veps}\) in \eqref{eq:min_probs} are equi-coercive with respect to the weak high-contrast convergence (by Lemma~\ref{lem:ts_key_lemma}) and a priori not the strong one. Since \(W^{1,2}(\Om;\sph^2)\) is closed under weak \(W^{1,2}\)-convergence, we always know that \(\tilde{m}\in \sph^2\) almost everywhere. However, the same is in general not true for the sum \(\tilde{m}+w\). This is paramount to the fact that \((m_k)_k\) cannot converge strongly to \(\tilde{m}+w\) in \(L^2(\Om;\R^3)\). 
    
    This is due to possible oscillation phenomena, not occurring within  the different cells of scale \(\veps\), but rather between these cells. To illustrate this phenomenon, consider any arbitrary non-zero smooth function \(\phi\in C^\infty(Q;\sph^2)\) that is equal to a constant vector \(s\in \sph^2\) outside a compact set \(K\subset Q_0\subset Q\). Then tessellate the domain \(\Om\) by cubes of size \(\veps Q\) with alternating copies of \(\phi\) and \(-\phi\). For \(\veps\to 0\) the resulting magnetic field will weakly two-scale converge to the average (cf. \cite[Chapter~2.3]{CioDo99})
    \begin{equation*}
        m(x,z)\coloneqq
        \begin{cases}
            0\quad &\text{for } z\in K,\\
            s\quad &\text{else},
        \end{cases}
    \end{equation*}
    which can also be written
    \begin{equation*}
        m(x,z)=\tilde{m}(x) + w(x,z)
        \qquad \text{for}\quad 
        \tilde{m}(x) \coloneqq s,
        \quad w(x,z)\coloneqq 
        \begin{cases}
            -s\quad &\text{for } z\in K,\\
            0\quad &\text{else}.
        \end{cases}
    \end{equation*}
    In particular, \(\tilde{m}\in \sph^2\) in \(\Om\), but it is not true that \(m=\tilde{m}+w\in \sph^2\) almost everywhere in \(\Om\). 
This latter remark suggests that a further relaxation of the energy needs to be undertaken when identifying the $\Gamma$-limit of the functionals $\mathcal{G}_{\veps}$ in the weak high-contrast sense, in the spirit of the analysis performed in \cite{Pi04}. However, this question is beyond the analysis in this paper and will be the subject of a forthcoming contribution.

The paper is organized as follows. In Section~\ref{sec:prelims} we fix some notation, collect the essential properties of two-scale convergence, describe the precise geometries for which the \(\Gamma\)-limit result in Theorem~\ref{thm:main_gamma_conv} holds, and recall a special extension operator that preserves the norm constraint, and which will be used in the proof of Theorem~\ref{thm:main_gamma_conv}. In Section~\ref{sec:main_approx_proof} we prove the approximation Theorem~\ref{thm:strong_ts_approx} in a more general way, already paving the way to apply it in the proof of Theorem~\ref{thm:main_gamma_conv}. Section~\ref{sec:main_gamma_proof} is devoted to the proof of Theorem~\ref{thm:main_gamma_conv}. 

\section{Preliminary results}\label{sec:prelims}
Throughout this paper we denote by \(\Om\subset \R^3\) an open and bounded set with Lipschitz boundary, and by \(Q:= (0,1)^3\subset \R^3\) the reference unit cell. 
For \(p\in (1,\infty)\) we denote by \(q\coloneqq\frac{p}{p-1}\) the dual exponent of \(p\). As usual, we set \(q\coloneqq\infty\) for \(p=1\).  
We use classical notation for Lebesgue and Sobolev spaces. We write
$W^{1,2}(\Omega)/\R$ to highlight the fact that a specific map is defined up to a real constant.
We will denote with the subscript \(\per\) functions on the unit cube \(Q\) that have periodic trace on the sides of \(Q\). The space \(W^{1,2}_{\per}(Q;\R^3)\) coincides with the closure of \(C^\infty_\per (Q;\R^3)\) in the \(W^{1,2}\)-norm (see for example \cite{CioDo99} for classical properties of periodic Sobolev functions). The characteristic functions of \(\Omzeps\) and \(\Omoeps\), equal to one on the respective domain and vanishing outside of them, will be denoted by 
\begin{equation}\label{eq:def_char}
	\chi_\veps^i(x)= 
	\begin{cases}
	1\quad &\text{if } x\in \Om_\veps^i,\\
	0 \quad & \text{else},
	\end{cases}
	\qquad \text{ for } i\in\{0,1\}.
\end{equation}
Note that due to definition \eqref{eq:setup2} such characteristic functions are in general not just restrictions of periodic functions to \(\Om\). In the case that we work with a fixed infinitesimal sequence \((\veps_k)_k\), we will refer to the associated characteristic functions by \(\chi_k^i:=\chi_{\veps_k}^i\). In the same style we will denote by \(\chi^0\) and \(\chi^1\) the characteristic functions of \(Q_0\) and \(Q_1\).

\subsection{Two-scale convergence}\label{sec:ts}
We recall here some classical results on two-scale convergence. The reader familiar with this topic might skip most of this subsection and proceed directly to Lemma \ref{lem:ts_key_lemma}. 

Originally introduced in \cite{Ng89} and further developed in \cite{Al92}, two-scale convergence is a notion of convergence that is able to capture fine-scale periodic oscillations and preserve the information about the oscillations in an additionally emerging micro-scale variable. It naturally allows to pass to the limit in special integral expressions involving products of rapidly oscillating functions. 
\begin{definition}[Weak two-scale convergence (\cite{Ng89,Al92,LNW02})]\label{def:weak_ts}
    Let \(p\in (1,\infty)\). We say that a bounded sequence of functions \(\sequence{u}{\veps}\subset L^p(\Om;\R^n)\) {\it weakly two-scale converges} in \(L^p(\Om;\R^n)\) to some limit \(u\in L^p(\Om\times Q;\R^n)\) if
		\begin{equation}\label{eq:def_ts}
			\int_\Om u_{\veps}(x)\cdot \psi\left(x,\frac{x}{{\veps_n}}\right)\, {\dl x}
			\to \int_\Om\int_Q u(x,z)\cdot\psi(x,z)\, {\dl z}{\dl x}
			\quad \forall \psi\in L^q(\Om;C_{per}(Q;\R^n)).
		\end{equation}
    We write \(u_\veps\wtwos u\).
\end{definition}
The reason why \(L^q(\Om;C_{per}(Q;\R^n))\) qualifies as the space of test functions in \eqref{eq:def_ts} is twofold. On the one hand, all functions \(\psi\) in this space are of {Carathéodory} type and thus \(\psi(x,x/\veps)\) is again a well-defined measurable function in \(L^2(\Om;\R^n)\); on the other hand, these functions are dense in \(L^q(\Om\times Q;\R^n)\) and {strongly two-scale} converge in the following sense. 
\begin{definition}[Strong two-scale convergence]\label{def:strong_ts}
    Let \(p\in (1,\infty)\) and \(\sequence{u}{\veps}\subset L^p(\Om;\R^n)\) be bounded. Then \(\sequence{u}{\veps}\) {\it strongly two-scale converges} in \(L^p(\Om;\R^n)\) to some \(u\in L^p(\Om\times Q;\R^n)\), and we write \(u_\veps\stwos u\), if it converges weakly two-scale to \(u\) and additionally there holds
    \begin{equation}
        \norm{u_\veps}{L^p(\Om;\R^n)}
        \to \norm{u}{L^p(\Om\times Q;\R^n)}.
    \end{equation}
\end{definition}
A connection between two-scale convergence and $L^p$-convergence is provided by the so-called {unfolding operator}.
\begin{definition}[Unfolding operator (\cite{CDG02, Vi04, Vi06})]\label{def:infold_op}
    For \(p\in [1,\infty]\) and \(\veps>0\) the {\it unfolding operator} 
    \begin{equation}
        S_\veps: L^p(\Om;\R^n)\to L^p(\R^3\times Q;\R^n)
\end{equation}
is defined as
\begin{equation}\label{eq:def_unfold_op}
	S_\veps(u)(x,z)
	:= \hat{u}\left(\veps\bigg\lfloor \frac{x}{\veps}\bigg\rfloor +\veps z\right),
\end{equation}
where \(\hat{u}\) denotes the extension of \(u\) to \(\R^3\setminus \Om\) by \(0\). 
\end{definition}
It follows immediately from the definition of the unfolding operator that \(S_\veps(u+ v)=S_\veps(u)+ S_\veps( v)\) and \(S_\veps(u\cdot w)=S_\veps(u)\cdot S_\veps(w)\).
A direct computation shows that \(S_\veps\) is a (non-surjective) linear isometry \(L^p(\Om;\R^n)\to L^p(\R^n\times Q;\R^n)\) for any \(p\in [1,\infty]\) (cf. \cite[Lemma 1.1]{Vi06}). In particular, for \(p\in [1,\infty)\) there holds
\begin{equation}\label{eq:isometry}
    \int_{\Om} |u(x)|^p\,{\dl x}
    = \int_{\Om}\int_Q |S_{\veps}(u)(x,z)|^p\,{\dl z}{\dl x}.
\end{equation}
Note that for any function \(\psi\in L^q(\Om;C_\per(Q;\R^n))\) and for \(\psi_\veps(x)\coloneqq \psi(x,x/\veps)\) we have by the periodicity in the second variable that
\begin{equation}\label{eq:special_unfold}
	S_\veps \left(\psi_\veps\right)(x,z)
	= \psi\left(\veps\bigg\lfloor \frac{x}{\veps}\bigg\rfloor+\veps z, \bigg\lfloor \frac{x}{\veps}\bigg\rfloor+z\right)
	= \psi\left(\veps\bigg\lfloor \frac{x}{\veps}\bigg\rfloor+\veps z,z\right).
\end{equation}
 Furthermore, the unfolding operator enjoys the nice property 
\begin{equation}
\label{eq:grad-unf}
    S_\veps(\veps \nabla u)
    = \nabla_z S_\veps(u)
    \qquad \forall u\in W^{1,p}(\Om;\R^n).
\end{equation}
The crucial insight connecting two-scale convergence with the unfolding operator is provided by the following result (cf. \cite[Proposition 2.5, Proposition 2.7]{Vi06}). 
\begin{proposition}\label{prop:wts_unfold}
     Let \(p\in (1,\infty)\). A bounded sequence \(\sequence{u}{\veps}\subset L^p(\Om;\R^n)\) weakly two-scale converges in \(L^p(\Om;\R^n)\) to \(u\in L^p(\Om\times Q;\R^n)\) iff
     \begin{equation}
         S_\veps(u_\veps) \weakly u 
         \quad \text{in } L^p(\R^3\times Q;\R^n).
     \end{equation}
     Analogously, \(\sequence{u}{\veps}\) strongly two-scale converges in \(L^p(\Om;\R^n)\) to \(u\) iff
     \begin{equation}
         S_\veps(u_\veps) \to u 
         \quad \text{in } L^p(\R^3\times Q;\R^n).
     \end{equation}
\end{proposition}

Proposition \ref{prop:wts_unfold} can be taken as a motivation to extend two-scale convergence to the limit cases \(p\in \{1,\infty\}\).
\begin{definition}\label{def:twos_p_one}
    Let \(p\in \{1,\infty\}\). A bounded sequence \(\sequence{u}{\veps}\subset L^p(\Om;\R^n)\) converges weakly (weakly star for \(p=\infty\)), respectively strongly, in \(L^p(\Om;\R^n)\)  to some limit \(u\in L^p(\Om\times Q;\R^n)\) if \((S_\veps(u_\veps)_\veps)_\veps\) converges weakly (weakly star), respectively strongly, to \(u\) in \(L^p(\R^3\times Q;\R^n)\).
\end{definition}

From Proposition \ref{prop:wts_unfold} and Definition \ref{def:twos_p_one} together with \eqref{eq:isometry} and \eqref{eq:special_unfold} one immediately gets that every purely oscillating function converges strongly two-scale. To be precise, for \(p\in [1,\infty]\) and every \(u\in L^p_\per(Q;\R^n)\) there holds \(u(x/\veps)\stwos u(z)\) in \(L^p(\Om;\R^n)\). 
With similar arguments one can see that for \(i\in \{0,1\}\) also the characteristic functions \(\chi_\veps^i=\chi_{\Om_\veps^i}\) defined in \eqref{eq:def_char} fulfil
\begin{equation}\label{eq:stwos_char}
    \chi_\veps^i(x)
    \stwos \chi^i(z)= \chi_{Q_i}(z)
    \qquad \text{in } L^p(\Om;\R)
\end{equation}
for all \(p\in [1,\infty)\). 

Eventually, Proposition \ref{prop:wts_unfold} also implies that for two sequences \(\sequence{u}{\veps}\subset L^p(\Om;\R^n),\sequence{v}{\veps}\subset L^q(\Om;\R^n)\) with \(u_\veps\wtwos u\) and \(v_\veps
\stwos v\) there holds \(u_\veps\cdot v_\veps\wtwos u\cdot v\) in \(L^1(\Om;\R^n)\).

We collect some further properties of two-scale convergence. 
The next lemma explains why two-scale convergence is often considered as an intermediate convergence between the strong and weak one in \(L^p(\Om\times Q;\R^n)\) (cf. \cite[Theorem 1.3]{Vi06}). 
\begin{lemma}\label{lem:ts_intermediate}
    Let \(p\in [1,\infty)\), let \(\sequence{u}{\veps}\subset L^p(\Om\times Q;\R^n)\) be bounded and \(u\in L^p(\Om;\R^n)\). 
		\begin{enumerate}[(i)]
			\item\label{item:strong_int_ts}
				If \(u\) is independent of the second variable and \(u_\veps\to u\) strongly in \(L^p(\Om;\R^n)\), then also \(u_\veps\stwos u\) in \(L^p(\Om;\R^n)\).
			\item\label{item:ts_int_weak}
				If \(u_\veps\overset{2s}{\rightharpoonup} u\) in \(L^p(\Om;\R^n)\), then 
				\begin{equation*}
					u_\veps\rightharpoonup \tilde{u}(x):= \int_Q u(x,y)\ dx \quad
					\text{weakly in } L^p(\Om;\R^n)
				\end{equation*}
			    and
                \begin{equation}
                    \norm{\tilde{u}}{L^p(\Om;\R^n)}
                    \le \norm{u}{L^p(\Om\times Q;\R^n)}
                    \le \liminf_{\veps} \norm{u_\veps}{L^p(\Om;\R^n)}.
                \end{equation}
		\end{enumerate}
\end{lemma}

Next, we present an augmented two-scale compactness result, which is by now well-known in the context of high-contrast homogenization, cf. \cite{ChCh12,DaKrPa21,DaGaPa2}.  
Note that, up to the weak convergence of the extensions, a similar lemma was already derived in \cite[Lemma 4.7]{Al92}. For convenience of the reader, we have included a proof of Lemma \ref{lem:ts_key_lemma} in Appendix \ref{sec:app_ts}.

\begin{lemma}\label{lem:ts_key_lemma}
    Let \(\sequence{\veps}{k}\) be an infinitesimal sequence and let \(\sequence{m}{k}\subset W^{1,2}(\Om;\R^3)\) be such that \(\sequence{m}{k}\) is bounded in \(L^2(\Om;\R^3)\). Let \(\sequence{\tilde{m}}{k}\subset W^{1,2}(\Om;\R^3)\) denote a sequence of extensions fulfilling
	\begin{gather*}
		\tilde{m}_k = m_k \quad \text{a.\,e. in } \Omoeps,\\
		\norm{\tilde{m}_k}{L^2(\Om;\R^3)} \le C_1\norm{m_k}{L^2(\Omoeps;\R^3)},\\
		\norm{D\tilde{m}_k}{L^2(\Om;\R^{3\times 3})} 
        \le C_1\norm{D m_k}{L^2(\Omoeps;\R^{3\times 3})}
	\end{gather*}
     for some constant \(C_1>0\) independent of \(\veps\). Moreover, assume that there exists another constant \(C_2>0\) such that
    \begin{equation}\label{eq:key_est}
        \norm{\veps_k\chi_{k}^0\nabla m_k}{L^2(\Om;\R^{3\times 3})} 
        + \norm{\chi_{k}^1\nabla m_k}{L^2(\Om;\R^{3\times 3})} 
        \le C_2.
    \end{equation}
    Then, there exist
    \begin{equation}\label{eq:elements}
        m\in L^2(\Om;W_{\per}^{1,2}(\R^3;\R^3)),\quad
        \tilde{m}\in W^{1,2}(\Om;\R^3),\quad 
        w\in L^2(\Om;W_0^{1,2}(Q_0;\R^3))
    \end{equation}
   such that, up to subsequences,
    \begin{alignat}{2}
			&m(x,z) = \tilde{m}(x) + w(x,z)
			\qquad &&\text{ for a.e. } (x,z)\in \Om\times Q, \label{eq:decomoposition} \\
			&m_k \wtwos m,
			\quad \veps _k\nabla m_k \wtwos \nabla_z w
			\qquad &&\text{ weakly two-scale in } L^2(\Om;\R^{3\times 3}), \label{eq:ts_convergences}\\
			&\tilde{m}_k \weakly \tilde{m}
			\qquad &&\text{ weakly in } W^{1,2}(\Om;\R^3). \label{eq:weak_conv_ext}
    \end{alignat}
\end{lemma}

Finally, we recall a two-scale result for the demagnetizing field defined in \eqref{eq:demag} and \eqref{eq:maxwell} (cf. \cite[Proposition4.3]{AlDF15}, \cite[Proposition~11]{SR07}).
\begin{lemma}\label{lem:demag_limit}
    Let \(\sequence{m}{k}\subset W^{1,2}(\Om;\R^3)\) weakly two-scale converge in \(L^2(\Om;\R^3)\) to a limit function \(m\in L^2(\Om\times Q;\R^3)\). Then, up to a subsequence, the sequence \((h_d[m_k])_k\) of associated demagnetizing fields weakly two-scale converges to
    \begin{equation*}
        h_d[\mv{m}{Q}](x)+h_d^z[m](x,z)\chi_\Om(x),
    \end{equation*}
    where 
    \begin{equation}
        h_d^z: L^2(\Om\times Q)\to L^2(\Om;L^2_\per(Q))
    \end{equation}
    is defined as \(h_d^z[m]\coloneqq \nabla_z r_m\),
    where \(r_m\in L^2(\Om;W^{1,2}_\per(Q)/\R)\) is such that for almost every \(x\in \Om\) the scalar function \(r_m(x,\cdot)\) is the unique solution in \(W^{1,2}_\per(Q)/\R\) to the cell problem
    \begin{equation}
        \int_Q \nabla_z r_m(x,z)\cdot \nabla_z \psi(z)\,{\dl z}
        = - \int_Q m(x,z)\cdot \nabla_z \psi(z)\,{\dl z}
        \quad \forall \psi\in W^{1,2}_\per(Q).
    \end{equation} 
\end{lemma}
We note that the above result exemplifies the abstract result in \cite{FoKr10} about two-scale convergence under differential constraints. Recalling \eqref{eq:diff_constraints}, it follows from \cite[Theorem~1.2]{FoKr10} that the limit of \((h_d[m_k])_k\) has to fulfil
    \begin{equation*}
    \mathcal{A}_z(m, h_d^z[m])
    \coloneqq
    \left(\begin{array}{cc}
       \diver_z  &  \diver_z\\
        0  &     \curl_z
    \end{array}\right)
    \left(\begin{array}{cc}
         m \\
         h_d^z[m]
    \end{array}\right)
    =0
    \quad \text{for a.e. } x\in \Om
    \qquad ,
\end{equation*}
which is exactly an equivalent formulation of \eqref{eq:diff_constraints_z}.

\subsection{Extension operators}\label{subsec:extensions}
Extension operators providing separate bounds on the $L^2$-norm of the maps and of their gradients are a fundamental tool in the study of homogenization problems on perforated domains. The seminal result in this direction can be found in \cite[Theorem 2.1]{AcPiMaPe92} and reads as follows. 
\begin{proposition}\label{thm:reflect_ext}
	Let \(1\le p<\infty\) and let \(\Omoeps\) be as in \eqref{eq:setup2}.
	Then, there exist a linear and continuous extension operator
	\begin{equation*}
		T_\veps: 
		W^{1,p}(\Omoeps;\R^3) \to W^{1,p}(\Om;\R^3)
	\end{equation*}
	and a constant \(C>0\) independent of \(\veps\) and \(\Om\), such that 
	\begin{gather*}
		T_\veps f = f \quad \text{a.\,e. in } \Omoeps,\\
		\norm{T_\veps f}{L^p(\Om;\R^3)} \le C\norm{f}{L^p(\Omoeps;\R^3)},\\
		\norm{D(T_\veps f)}{L^p(\Om;\R^{3\times 3})} \le C\norm{Df}{L^p(\Omoeps;\R^{3\times 3})}
	\end{gather*}
	for every \(f \in W^{1,p}(\Omoeps;\R^3)\). 
\end{proposition}
In our problem we will need a variants of the above result in order to take care of extensions of the magnetizations. The main hurdle concerns the necessity of extending the maps while at the same time preserving the saturation constraint. This is guaranteed by \cite[Theorem 3.1]{GaHaPa23},
whose simplified statement in our setting is stated in the next proposition. 
\begin{proposition}\label{thm:target_ext}
	Let \(\Omoeps\) be as in \eqref{eq:setup2}.
	There exists an extension operator
	\begin{equation*}
	\sfT_\veps: 
	W^{1,2}(\Omoeps;\sph^2) \to W^{1,2}(\Om;\sph^2)
	\end{equation*}
	and a constant \(C>0\), independent of \(\veps\), such that for \(\veps\) sufficiently small
	\begin{gather*}
	\sfT_\veps f = f\quad \text{a.\,e. in } \Omoeps, \label{eq:id_ext}\\
	\norm{\sfT_\veps f}{L^2(\Om;\sph^2)} \le C\norm{f}{L^2(\Omoeps;\sph^2)}, \label{eq:funct_ext}\\
	\norm{D(\sfT_\veps f)}{L^2(\Om;\R^{3\times 3})} \le C\norm{Df}{L^2(\Omoeps;\R^{3\times 3})}, \label{eq:grad_ext}
	\end{gather*}
	for every \(f\in W^{1,2}(\Omoeps;\sph^2)\).
\end{proposition}

\section{Proof of Theorem~\ref{thm:strong_ts_approx}}\label{sec:main_approx_proof}

We preliminary recall an approximation result proving that every function in the space \(L^2(\Om;W^{1,2}_0(Q_0;\R^3))\) can be attained as a strong two-scale limit of smooth functions in \(C_c^\infty(\Om;\R^3)\). The idea of the construction goes back to  \cite[Lemma~6.1]{DaGaPa2}, as well as \cite[Lemma~22 and Proposition~17]{ChCh12}. To make the presentation as self-contained as possible we also provide a short proof.  

\begin{lemma}\label{lem:strong_l2}
    For any infinitesimal sequence \((\veps_k)_k\) and any \(w\in  L^2(\Om;W_0^{1,2}(Q_0;\sph^2))\) there exists a sequence \((\hat{w}_k)_k\subset C_c^\infty(\Om;\R^3)\) such that $\hat{w}_k \chi^1_k=0$ for every $k$, with
    \begin{equation}\label{item:stwos}
        \hat{w}_k\stwos w(x,z)
        \qquad \text{in } L^2(\Om;\R^3),
    \end{equation}
	and 
	\begin{equation}\label{item:stwos_grad}
		\veps_k \nabla \hat{w}_k
            \stwos \nabla_zw(x,z)
        \qquad \text{in } L^2(\Om;\R^3).
	\end{equation}
\end{lemma}
\begin{proof}
Let \(w\in L^2(\Om;W_0^{1,2}(Q_0;\sph^2))\). We approximate \(w\) by functions \(\sequence{w}{j}\subset C_c^\infty(\Om\times Q_0;\R^3)\). Then, we define 
\begin{equation}\label{eq:cubes_av}
w^j_k(x,z)
\coloneqq 
\begin{cases}
\frac{1}{\veps_k^3}\int_{\veps_k(t+Q)}w_j(\bar{x},z)\,{\dl\bar{x}}\qquad  &\text{if } x\in \veps_k(t+Q) \text{ for some } t\in \hat{Z}_k, \\
0\qquad  &\text{else}
\end{cases}
\end{equation}
by averaging on cubes whose side lengths are associated to the sequence \(\sequence{\veps}{k}\).
In the expression above, 
$\hat{Z}_k:=\{t\in \mathbb{Z}^3:\,\veps_k(t+Q)\subset \Omega\}$.

Note that since \(w_j\) has compact support in \(Q\) in the \(z\)-variable for all \(j\in \N\), the functions
\begin{equation}\label{eq:make_osc}
    v_k^j(x)\coloneqq w_k^{j}\left(x,\frac{x}{\veps_k}\right)
\end{equation}
are again elements of \(C_c^\infty (\Om)\) for all \(j,k\in \N\). Furthermore, for any fixed \(j\in \N\) the sequence \((v_k^j)_k\) strongly two-scale converges to \(w_j\) for \(k\to \infty\); and at the same time \((\veps_k \nabla v_k^j)_k\) strongly two-scale converges to \(\nabla _z w_j\). This follows from the smoothness of the functions and the mean value theorem (compare \cite[Lemma~6.1]{DaGaPa2}). Eventually, by construction, $v^j_k\chi^1_k=0$ for every $k\in \N$.

In particular, owing to Proposition \ref{prop:wts_unfold}, we find
$$\lim_{j\to +\infty}\lim_{k\to +\infty}\left(\|S_k(v^j_k)-w\|_{L^2(\R^3\times \Omega;\R^3)}+\|\nabla_z S_k(v^j_k)-\nabla_z w\|_{L^2(\R^3\times \Omega;\R^{3\times 3})}\right)=0.$$
By Attouch's diagonalization lemma \cite[Lemma 1.15 and Corollary 1.16]{Att86} we find a subsequence $(j(k))_k$ such that
$$\lim_{k\to +\infty}\left(\|S_k(v^{j(k)}_k)-w\|_{L^2(\R^3\times \Omega;\R^3)}+\|\nabla_z S_k(v^{j(k)}_k)-\nabla_z w\|_{L^2(\R^3\times \Omega;\R^{3\times 3})}\right)=0.$$
The thesis follows then by Proposition \ref{prop:wts_unfold}, setting $\hat{w}_k:=v^{j(k)}_k$ for all $k\in \N$.
\end{proof}

Let now $\tilde{m}\in W^{1,2}(\Omega;\sph^2)$ and $w\in L^2(\Omega;W^{1,2}_0(Q_0;\R^3))$ be such that
\begin{equation*}
    \tilde{m}(x)+w(x,z)\in \sph^2 
    \quad \text{for a.e. }
    (x,z)\in \Omega\times Q_0.
\end{equation*}
 Then, we might approximate \(\tilde{m}\) by functions 
 \begin{equation}
\label{eq:mtildek}
\sequence{\tilde{m}'}{k}\subset C^\infty(\Om;\R^3)\quad\text{such that}\quad\tilde{m}'_k\to \tilde{m}\quad\text{strongly in }W^{1,2}(\Omega;\R^3),
\end{equation}
 and we apply Lemma \ref{lem:strong_l2} to find a sequence \((\hat{w}_k)_k \subset C_c^\infty(\Om;\R^3)\) that fulfils \eqref{item:stwos} and \eqref{item:stwos_grad}. In particular, 
 \begin{equation}\label{eq:pre_pre_sequence}
	v_k(x):=\tilde{m}'_k(x) +\hat{w}_k(x) 
	\quad \in C^\infty(\Om;\R^3)\cap W^{1,2}(\Om;\R^3)
\end{equation}
and
 \begin{alignat*}{2}
          v_k
          &\stwos \tilde{m} + w\chi^0
          \qquad &&\text{strongly two-scale in } L^2(\Om;\R^3)\\
          \veps_k \nabla v_k
          &\stwos (\nabla_z w)\chi^0 
          \qquad &&\text{strongly two-scale in } L^2(\Om;\R^3).
 \end{alignat*}
In order to construct a recovery sequence, it would be natural to project back the functions in \eqref{eq:pre_pre_sequence} onto \(\sph^2\) by means of the map
	\begin{equation}\label{eq:projection}
            \pi:\R^3\setminus\{0\}\to \sph^2,
            \quad
		\pi(x)
		= \frac{x}{\norm{x}{}},
	\end{equation}
 which coincides with the nearest point projection on every tubular neighbourhood of \(\sph^2\). Unfortunately, in general, it is not guaranteed that the pointwise projections of the maps in \eqref{eq:pre_pre_sequence} are well defined. In particular, in principle, sets of positive measure could be mapped by $v_k$ onto the origin, so that the composition $\pi\circ v_k$ could have singularities on a set of positive measure. 
 To circumvent this possible pitfall, we adapt the strategy in \cite[Theorem 4]{BeZh88}. We prove now a slight generalization of Theorem~\ref{thm:strong_ts_approx}, adding arbitrary perturbations \(\veps\psi\), for \(\psi \in C_c^\infty(\Om;C^\infty_{\per}(Q;\R^3))\), to the approximating sequence. This additional degree of freedom will be exploited in the proof of Theorem~\ref{thm:main_gamma_conv} in Section~\ref{sec:main_gamma_proof}. Moreover, Theorem~\ref{thm:strong_ts_approx} follows from Theorem~\ref{thm:recov_mag} simply be setting \(\psi\equiv 0\).
 
\begin{theorem}\label{thm:recov_mag}
Let $\tilde{m}\in W^{1,2}(\Omega;\sph^2)$ and $w\in L^2(\Omega;W^{1,2}_0(Q_0;\R^3))$ be such that 
\begin{equation*}
    \tilde{m}(x)+w(x,z)\in \sph^2 
    \quad \text{for a.e. }
    (x,z)\in \Omega\times Q_0.
\end{equation*}  
Then, for every \(\psi\in C_c^\infty(\Om;C^\infty_{\per}(Q;\R^3))\) there exists a sequence of functions \(\sequence{m}{k}\subset W^{1,2}(\Om;\sph^2)\) such that
    \begin{alignat}{2}
        &m_k 
        \stwos \tilde{m}(x) + w(x,z)\chi^0(z)
        && \quad \text{in } L^2(\Om;\sph^2),\label{eq:first_stwos}\\
        &\veps_k \nabla m_k
        \stwos \nabla_z w(x,z)\chi^0(z)
        && \quad \text{in } L^2(\Om;\R^{3\times 3}),\label{eq:second_stwos}\\
        & \nabla m_k\chi_k^1
        \stwos \nabla \tilde{m}(x) + \nabla_z\psi(x,z)\cdot\nabla \pi(\tilde{m}(x))\chi^1(z)
        && \quad \text{in } L^2(\Om;\R^{3\times 3}).\label{eq:third_stwos}
    \end{alignat}
    Moreover, there holds
    \begin{equation*}
        m_k\in C^\infty(\Om\setminus \{a_{k_1},\dots, a_{k_l}\};\sph^2),
    \end{equation*}
    i.e., the functions \(m_k\) are smooth up to a finite number of points, where the points themselves and the number \(k_l\) depend on the index \(k\in N\).
\end{theorem}

\begin{proof}
For $\tilde{m}\in W^{1,2}(\Omega;\sph^2)$ and $w\in L^2(\Omega;W^{1,2}_0(Q_0;\R^3))$ as in the thesis, we choose approximating sequences \(\sequence{\tilde{m}'}{k}\subset C^\infty(\Om;\R^3)\) and \(\sequence{\hat{w}}{k}\subset C_c^\infty(\Om;\R^3)\) with the properties in \eqref{eq:mtildek} and in Lemma~\ref{lem:strong_l2}. For \(\psi\in C_c^\infty(\Om;C^\infty_{\per}(Q;\R^3))\) we define (modifying \eqref{eq:pre_pre_sequence})
 \begin{equation}\label{eq:pre_sequence}
	v_k(x):=\tilde{m}'_k(x) + \veps_k\psi(x,x/\veps_k) + \hat{w}_k(x) 
	\quad \in C^\infty(\Om;\R^3)\cap W^{1,2}(\Om;\R^3).
\end{equation}
There holds
 \begin{alignat}{2}
          v_k
          &\stwos \tilde{m} + w\chi^0
          \qquad &&\text{strongly two-scale in } L^2(\Om;\R^3)\label{eq:unproj_stwos}\\
          \veps_k \nabla v_k
          &\stwos (\nabla_z w)\chi^0 
          \qquad &&\text{strongly two-scale in } L^2(\Om;\R^3).\label{eq:unproj_stwos_grad}
 \end{alignat}
 Note that we can always assume that \(S_k(v_k)\) converges a.e. to \(\tilde{m} + w\chi^0\) in \(\R^3\times Q\) by extracting a proper subsequence.

We follow now the strategy outlined in \cite[Theorem 4]{BeZh88}. For $\delta\in (0,1)$, let $\sph^2_{1-\delta}$ and $B_{1-\delta}$ be the sphere and ball centered in the origin and with radius $1-\delta$.
 
First, observe that by Sard's theorem and \cite[Theorem 5.22]{Lee03} the preimages \(F_{k,\delta}\coloneqq v_k^{-1}(\sph_{1-\delta}^2)\) in \(\Om\) are  submanifolds of co-dimension one in \(\Om\) for almost every \(0<\delta<1/2\). They are also the topological boundary in \(\Om\) of the sets \(V_{k,\delta}\coloneqq v_k^{-1}(B_{1-\delta})\) for \(B_{1-\delta}=\{y\in \R^3:\, y\le 1-\delta\}\). In what follows, we denote by $D$ the set of $\delta\in (0,1/2)$ such that the above properties hold for all \(k\in \N\) and we fix $\delta\in D$.
 
In view of \cite[Theorem~4]{BeZh88} (see also \cite[Lemma~2.3]{BPV14} and \cite{GaHaPa23}), for every \(k\in \N\) there exists \(a_k\in B_{1/4}\) such that, defining the shifted projections
	\begin{equation}
		\pi^\delta(x)
		\coloneqq \pi(x)(1-\delta)
		= \frac{x}{\norm{x}{}}(1-\delta),
		\quad 
		\pi^{\delta}_{a_k}(x)
		\coloneqq \pi^{\delta}(x-a_k)
	\end{equation}
	and 
	\begin{equation}
 \label{eq:def-hat}
		\hat{\pi}^{\delta}_{a_k}
		\coloneqq (\pi^{\delta}_{a_k}|_{\sph^2_{1-\delta}})^{-1}\circ \pi^{\delta}_{a_k},
	\end{equation}
	we have 
	\begin{equation}\label{eq:sob_prop}
	\hat{\pi}^{\delta}_{a_k}\circ v_k
	\quad \in W^{1,2}(V_{k,\delta};\sph^2_{1-\delta})
	\end{equation}
	with 
	\begin{equation}\label{eq:key_est_proj}
		\norm{\nabla (\hat{\pi}^{\delta}_{a_k}\circ v_k)}{L^2(V_{k,\delta};\sph^2)}
		\le C \norm{\nabla v_k}{L^2(V_{k,\delta};\R^3)}.
	\end{equation}
        The important detail in the above construction lies in the fact that, again by Sard's theorem, for almost every \(a\in B_{1/4}\) the preimages \(v_k^{-1}(a)\) are points in \(V_{k,\delta}\subset \Om\) (of co-dimension \(3\)). Estimate \eqref{eq:key_est_proj} only holds true because the exponent \(p=2\) is smaller than the space dimension of the ambient Euclidean space \(\R^3\) of \(\sph^2\).
	Now, we set 
	\begin{equation}
		m_k
		\coloneqq 
		\begin{cases}
			\pi\circ \hat{\pi}^{\delta}_{a_k}\circ v_k\quad &\text{for } x\in V_{k,\delta} \\
			\pi\circ v_k\quad &\text{for } x\in \Om\setminus \overline{V_{k,\delta} }.
		\end{cases}
	\end{equation}
	Note that by \eqref{eq:def-hat} there holds \((m_k)_k\subset W^{1,2}(\Om;\sph^2)\). 
	We apply now the unfolding operator to prove the strong two-scale convergences in \eqref{eq:first_stwos} and \eqref{eq:second_stwos}.
    To this end, define 
    \begin{equation}
	W_{k,\delta}
	\coloneqq \{(x,z)\in \R^3\times Q:\, \veps_k\lfloor x/\veps_k\rfloor + \veps_k z\in V_{k,\delta}\}
	\end{equation}
    Recall that away from zero -- in particular, outside of \(B_{1-\delta}\) -- the nearest point projection \(\pi\) is a smooth function. Thus, outside $W_{k,\delta}$, arguing by unfolding, we find 
    \begin{align}
    \label{eq:926}
        \begin{split}
        &\int_{(\R^3\times Q)\setminus \ol{W_{k,\delta}}}
        |S_k(m_k) - [\tilde{m}+w\chi^0]|^2
        \,{\dl z}{\dl x}\\
        &=\int_{(\R^3\times Q)\setminus \ol{W_{k,\delta}}}
        |S_k(\pi(v_k)) - [\tilde{m}+w\chi^0]|^2
        \,{\dl z}{\dl x}\\
        &= \int_{(\R^3\times Q)\setminus \ol{W_{k,\delta}}}
        |\pi(S_k(v_k)) - \pi(\tilde{m}+w\chi^0(z))|^2 \,{\dl z}{\dl x}\\
        &\le \norm{\nabla \pi}{L^\infty(B_1\setminus {\ol{B_{1-\delta}}})}^2 \int_{\R^3}\int_Q
        |S_k(v_k) - [\tilde{m}+w\chi^0(z)]|^2 \,{\dl z}{\dl x},
        \end{split}
    \end{align}
    where the right-hand side converges to zero owing to \eqref{eq:unproj_stwos}. We used in the second equality that the unfolding operator commutes with the projection \(\pi\), since it has no influence on the values of a map.

    The same arguments apply to the sequence \((\veps_k \nabla m_k)_k\), upon noticing that
    \begin{equation*}
        \nabla_z w(x,z)
        = \nabla_z (\pi\circ (\tilde m(x)+w(x,z)))
        = \nabla_z w(x,z) 
        \cdot \nabla \pi(\tilde m(x)+w(x,z))
    \end{equation*}
    for almost every \( (x,z)\in \Om\times Q_0\),
    because \(\tilde{m}+w\in \sph^2\) for almost every \((x,z)\in \Om\times Q_0\). Indeed, we find
    \begin{align}\label{eq:321}
        \begin{split}
        &\int_{(\R^3\times Q)\setminus \ol{W_{k,\delta}}} |S_k(\veps_k\nabla m_k)-\nabla_z(w\chi^0)|^2\,{\dl z}{\dl x}\\
        &\quad=\int_{(\R^3\times Q)\setminus \ol{W_{k,\delta}}} |S_k(\veps_k\nabla(\pi(v_k)))-\nabla_z(w\chi^0)\cdot\nabla\pi(\tilde{m}+w\chi^0)|^2\,{\dl z}{\dl x}\\
        &\quad =\int_{(\R^3\times Q)\setminus \ol{W_{k,\delta}}} |S_k(\veps_k\nabla v_k\nabla \pi(v_k)))-\nabla_z(w\chi^0)\cdot\nabla\pi(\tilde{m}+w\chi^0)|^2\,{\dl z}{\dl x}\\
        &\quad =\int_{(\R^3\times Q)\setminus \ol{W_{k,\delta}}} |S_k(\veps_k\nabla v_k)\cdot S_k(\nabla \pi(v_k)))-\nabla_z(w\chi^0)\cdot\nabla\pi(\tilde{m}+w\chi^0)|^2\,{\dl z}{\dl x}\\
        &\quad\leq \|\nabla \pi\|_{L^{\infty}(B_1\setminus\ol{B_{1-\delta}})}^2\int_{(\R^3\times Q)\setminus \ol{W_{k,\delta}}} |S_k(\veps_k\nabla v_k)-\nabla_z w\chi^0|^2\,{\dl z}{\dl x}\\
        &\qquad+\int_{(\R^3\times Q)\setminus \ol{W_{k,\delta}}} |\nabla_z w\chi^0\cdot( S_k(\nabla \pi(v_k))-\nabla\pi(\tilde{m}+w\chi^0))|^2\,{\dl z}{\dl x}.
        \end{split}
    \end{align}
    The first integral on the right-hand side of \eqref{eq:321} vanishes due to \eqref{eq:unproj_stwos_grad} and the second one by virtue of Severini-Egorov Theorem. In fact, according to the latter result, using that \(S_k(v_k)\) converges almost everywhere to \(\tilde{m} + w\chi^0\) and that \(|S_k(v_k)-(\tilde{m} + w\chi^0)|\le 2\) almost everywhere in \(\R^3\times Q\) , for every \(\mu>0\)  we have the estimate
    \begin{align*}
        &\int_{(\R^3\times Q)\setminus \ol{W_{k,\delta}}} |\nabla_z w\chi^0\cdot( S_k(\nabla \pi(v_k))-\nabla\pi(\tilde{m}+w\chi^0))|^2\,{\dl z}{\dl x}\\
        &\quad\le \norm{\nabla _z w}{L^2((\R^3\times Q)}^2
        \left(\norm{\nabla^2 \pi}{L^\infty(B_1\setminus\ol{B_{1-\delta}})}\mu
        +2|E_\mu|\norm{\nabla^2 \pi}{L^\infty(B_1\setminus\ol{B_{1-\delta}})}\right)
    \end{align*}
    for \(k\in \N\) sufficiently large and for some open subset \(E_\mu\subset (\R^3\times Q)\setminus \ol{W_{k,\delta}}\) such that \(|E_\mu|<\mu\). Letting \(\mu\to 0\) and \(k\to \infty\) we validate the claim of the theorem outside of \(W_{k,\delta}\).

 It remains to prove that the contributions from \(S_{k}(\pi\circ\hat{\pi}^{\delta}_{a_k}\circ v_k)\) and \(\nabla _z S_{k}(\pi\circ\hat{\pi}^{\delta}_{a_k}\circ v_k)\) on the set \(W_{k,\delta}\)
 become infinitesimally small in the limit. 
 We set
	\begin{equation}
		A_{k,\delta}
		\coloneqq\{(x,z)\in \R^3\times Q:\, |S_{k}(v_k)(x,z)-[\tilde{m}(x) + w(x,z)\chi^0(z)]|\ge \delta\}.
	\end{equation}
	From the strong \(L^2\)-convergence of \((S_k(v_k))_k\) in \eqref{eq:unproj_stwos} it follows that the Lebesgue measure of \(A_{k,\delta}\) converges to zero for \(k\to \infty\). Additionally,
 since $S_k(v_k)\in B_{1-\delta}$ almost everywhere on $W_{k,\delta}$ and $\tilde{m}+w\in \sph^2$ almost everywhere on $\Omega\times Q$, it follows that \(W_{k,\delta}\subset A_{k,\delta}\). Then, from the property that \(S_{k}(\pi\circ\hat{\pi}^{\delta}_{a_k}\circ v_k)\) takes values in \(\sph^2\), we infer that
	\begin{equation}
 \label{eq:extra4}
		\int_{W_{k,\delta}} |S_{k}(\pi\circ\hat{\pi}^{\delta}_{a_k}\circ v_k)(x,z)-[\tilde{m}(x) + w(x,z)\chi^0(z)|^2\,{\dl z}{\dl x}
		\le 4 |	A_{k,\delta}|,
	\end{equation} 
	which vanishes in the limit \(k\to \infty\). Turning to the gradients \(\nabla_z S_{k}(\pi\circ\hat{\pi}^{\delta}_{a_k}\circ v_k)\), we first infer from \eqref{eq:unproj_stwos_grad} the equiintegrability of the sequence
  \((S_k(\veps_k\nabla v_k))_k\). 
  From the smoothness of $\pi$ on $\sph^2_{1-\delta}$, and by the fact that the unfolding operator is an isometry, inequality \eqref{eq:key_est_proj} translates into
  \begin{align*}
      &\int_{W_{k,\delta}} |S_k(\nabla (\pi\circ\hat{\pi}^{\delta}_{a_k}\circ v_k))|^2\,{\dl z}{\dl x}=\int_{V_{k,\delta}} |\nabla (\pi\circ\hat{\pi}^{\delta}_{a_k}\circ v_k)|^2\,{\dl x}\\
      &\quad\le \|\nabla\pi\|_{L^{\infty}(\sph^2_{1-\delta})}\int_{V_{k,\delta}} |\nabla (\hat{\pi}^{\delta}_{a_k}\circ v_k)|^2\,{\dl x}\leq C\int_{V_{k,\delta}} |\nabla  v_k|^2\,{\dl x}.
  \end{align*}
Thus, by \eqref{eq:grad-unf},
\begin{align}
 &\nonumber\int_{W_{k,\delta}} |\nabla_zS_k(\pi\circ\hat{\pi}^{\delta}_{a_k}\circ v_k)|^2\,{\dl z}{\dl x}
      = \int_{W_{k,\delta}} |S_k(\veps_k\nabla(\pi\circ\hat{\pi}^{\delta}_{a_k}\circ v_k))|^2\,{\dl z}{\dl x}
      \le C\int_{V_{k,\delta}} |\veps\nabla v_k|^2\,{\dl x}\\
      &\quad=C\int_{W_{k,\delta}} |S_k(\veps\nabla v_k)|^2\,{\dl z}{\dl x}\to 0,\label{eq:extra1}
\end{align}
owing to \eqref{eq:unproj_stwos_grad}, and the fact that
\begin{equation}
    \label{eq:num}
    |W_{k,\delta}|\leq |A_{k,\delta}|\to 0
\end{equation}
as $k\to +\infty$. In particular, by \eqref{eq:extra1} and \eqref{eq:num}, we also infer that
\begin{equation}
    \label{eq:extra2}
    \|\nabla_z S_k(m_k)-\nabla_z w\chi^0\|_{L^2(W_{k,\delta};\R^{3\times 3})}\to 0
\end{equation}
as $k\to+\infty$. By \eqref{eq:926}, \eqref{eq:321}, \eqref{eq:extra4}, and \eqref{eq:extra2} we obtain \eqref{eq:first_stwos} and \eqref{eq:second_stwos}.

  It remains to identify the two-scale limit of \((\nabla m_k\chi_k^1)_k\) and to prove \eqref{eq:third_stwos}.
    Again, we split the integral over \(\nabla m_k\chi_k^1\) into the part on \(V_{k,\delta}\) and the one on the complement in \(\Om\). For the first contribution, since $\hat{w}_k\chi^1_k=0$ for every $k$, cf. Lemma \ref{lem:strong_l2}, we infer that $v_k(x)=\tilde{m}_k'(x)+\veps_k\psi\left(x,\frac{x}{\veps_k}\right)$ for almost every $x\in \Omega_k^1$. Thus, we find the estimate
    \begin{align*}
        &\|\nabla m_k \chi_k^1\|_{L^2(V_{k,\delta};\R^3)}\leq C\|\nabla \pi\|_{L^{\infty}(B_1\setminus B_{1-\delta})}\|\nabla V_{k,\delta}\chi_k^1\|_{L^2(V_{k,\delta};\R^3)}\\
        &\quad\leq C\|\nabla \pi\|_{L^{\infty}(B_1\setminus B_{1-\delta})}\left(\|\nabla\tilde{m}'_k\|_{L^2(V_{k,\delta};\R^{3\times 3})}+\veps_k\|\nabla_x \psi\|_{L^2(V_{k,\delta};\R^{3\times 3})}+\|\nabla_z \psi\|_{L^2(V_{k,\delta};\R^{3\times 3})}\right).
    \end{align*}
    From \eqref{eq:mtildek} and the fact that $|V_{k,\delta}|\to 0$ we hence obtain that
    $$\|\nabla m_k \chi_k^1\|_{L^2(V_{k,\delta};\R^3)}\to 0$$
    as $k\to+\infty$.

    Let us now set \(\hat{m}_k(x)\coloneqq m_k\chi_k^1(x)\). Then,
    \begin{equation}
        \hat{m}_k(x)
        = \pi(\tilde{m}'_k(x) + \veps_k\psi(x,x/\veps_k))
        \qquad \text{for a.\,e. }  x\in \Om\setminus\ol{V_{k,\delta}}.
    \end{equation} 
    Consequently, on \(\Omega\setminus V_{k,\delta}\) there holds 
    \begin{equation*}
        \nabla \hat{m}_k(x)
        = \left(\nabla \tilde{m}'_k(x) + \veps_k\nabla_x\psi(x,x/\veps_k) + \nabla_z\psi(x,x/\veps_k)\right)
        \cdot \nabla \pi( \hat{m}_k(x))\chi_k^1(x).
    \end{equation*}
    Recall now that the right-hand side in \eqref{eq:third_stwos} can be rewritten as
    \begin{equation*}
        \nabla \tilde{m}(x) + \nabla_z\psi(x,z)\cdot\nabla \pi(\tilde{m}(x))\chi^1(z)
        = (\nabla \tilde{m}(x) + \nabla_z\psi(x,z))\cdot\nabla \pi(\tilde{m}(x))\chi^1(z),
    \end{equation*}
    since \(\tilde{m}\in \sph^2\) for a.e. \(x\in \Om\).
    Employing once more the smoothness of \(\pi\) outside of \(B_{1-\delta}\), the claim follows then from 
    \begin{equation*}
        \nabla \tilde{m}'_k(x) + \veps_k\nabla_x\psi(x,x/\veps_k) + \nabla_z\psi(x,x/\veps_k)
        \stwos \nabla \tilde{m}(x) + \nabla_z\psi(x,z)
        \qquad \text{in } L^2(\Om;\R^{3\times 3}),
    \end{equation*}
    which results from the strong convergence of the approximations \(\sequence{\tilde{m}'}{k}\) towards \(\tilde{m}\) in \(W^{1,2}(\Om;\R^3)\) and the smoothness of \(\psi\in C_c^\infty(\Om;C_\per^\infty(Q;\R^3))\).

    The statement that the elements \(m_k\) of the approximating sequence are smooth up to a finite number of points, whose position and number depends on the index \(k\in \N\) is a direct consequence of the construction in \eqref{eq:def-hat}-\eqref{eq:sob_prop} (see \cite[Theorem~4]{BeZh88}). 
\end{proof}
\begin{remark}
    The proof of Theorem~\ref{thm:recov_mag} can be split into the following two steps: (i) first, using Lemma~\ref{lem:strong_l2}, we construct a sequence approximating \(\tilde{m}+w\) strongly two-scale (and similarly for the gradient), where, however, we forget about the \(\sph^2\)-constraint; (ii) we then use suitable projections in the target space to restore the constraint. In theory, one could also try to first find approximations, smooth up to finitely many points, of \(\tilde{m}+w\) (and its gradient) in \(L^p(\Om \times Q;\sph^2)\) with the help of the ideas in \cite[Theorem~4]{BeZh88}, and then to build strong two-scale approximations in the lines of Lemma~\ref{lem:strong_l2}. In fact, applying the approximation scheme from \cite[Theorem~4]{BeZh88} to a target function in \(L^2(\Om\times Q;\sph^2)\) is possible in theory, but there is a catch with following Lemma~\ref{lem:strong_l2} and it is two-fold: First, through the averaging in \eqref{eq:cubes_av} we fall again out of \(\sph^2\) in the target space; and second, introducing the fast oscillations in \eqref{eq:make_osc} might lead to ill-defined functions. The reason for the latter point is the following: In applying \cite[Theorem~4]{BeZh88} to a function in \(L^2(\Om\times Q;\sph^2)\) with smooth approximations in \(C^\infty(\Om\times Q;\sph^2)\), we need to choose a sequence of shifts \(\sequence{a}{k}\) for the projections such that the analog of \eqref{eq:sob_prop} is fulfilled. The preimages of the singularities of the shifted projections would have co-dimension \(3\) in \(\Om\times Q\), and nothing tells us a priori that they lie completely in \(\Om\), i.e., that, when projected onto \(Q\), they have co-dimension \(3\) in \(Q\) (in particular, if this fails, they are not of Carathéodeory type). But if we mimic \eqref{eq:make_osc} for such functions, we cannot assure that the resulting maps are well-defined. Hence, the path chosen for Theorem~\ref{thm:recov_mag} appears to the authors to be the only feasible option to construct the sought strong two-scale approximations. 
    \end{remark}
\section{Proof of theorem \ref{thm:main_gamma_conv}}\label{sec:main_gamma_proof}

\subsection{Splitting and analysis of the energy on the matrix component}\label{subsec:splitting}
We define for \(m\in W^{1,2}(\Om;\sph^2)\) the energies
\begin{equation}\label{eq:minus_self}
    \mathcal{F}_\veps(u,m)
    := \mathcal{G}_\veps(u,m) 
    - \int_{\R^3}\mathcal{W}(m)\,{\dl x},
\end{equation}
as well as
\begin{align}
    \begin{split}\label{eq:soft_energy}
        \mathcal{F}^0_\veps(m)
        &\coloneqq\frac{1}{2}\int_{\Omzeps} \veps ^2|\nabla m|^2\, {\dl x},
    \end{split}\\
    \begin{split}\label{eq:stiff_energy}
        \mathcal{F}^1_\veps(m)
        &\coloneqq \frac{1}{2}\int_{\Omoeps} |\nabla m|^2\, {\dl x}.
    \end{split}
\end{align}

The next proposition will allow to separately study the \(\Gamma\)-limits of \(\sequence{\mathcal{F}^0}{\veps}\) and \(\sequence{\mathcal{F}^1}{\veps}\).

\begin{lemma}[Splitting]\label{prop:splitting}
    For \((\veps_k)_k\) an infinitesimal sequence let \(\sequence{m}{k}\subset W^{1,2}(\Om;\sph^2)\) converge to some \((\tilde{m},w)\in W^{1,2}(\Om;\R^3)\times L^2(\Om;W_0^{1,2}(Q_0;\R^3))\) in the high-contrast sense of Definition~\ref{def:conv_for_gamma}.
    Let further \((\tilde{m}_k)_k\) be the extensions in \(W^{1,2}(\Om;\sph^2)\) provided by Proposition~\ref{thm:target_ext}. We set \(w_k:=m_k-\tilde{m}_k\). Then
    \begin{align}
        \begin{split}\label{eq:liminf_superadd}
        \liminf_{k\to +\infty}
        \mathcal{F}^0_\veps(w_k) + 
        \liminf_{k\to +\infty}
        \mathcal{F}^1_\veps(\tilde{m}_k)
        \le \liminf_{k\to +\infty}
        \mathcal{F}_\veps(m_k),
        \end{split}\\
        \begin{split}\label{eq:limsup_subadd}
         \limsup_{k\to +\infty}
        \mathcal{F}_\veps(m_k)
        \le \limsup_{k\to +\infty}
        \mathcal{F}^0_\veps(w_k) +
        \limsup_{k\to +\infty}
        \mathcal{F}^1_\veps(\tilde{m}_k).  
        \end{split}
    \end{align}
\end{lemma}
\begin{proof}
    A direct computation shows that
    \begin{align*}
        \mathcal{F}_\veps(m_k)
        = \mathcal{F}^0_\veps(m_k)
        + \mathcal{F}^1_\veps(m_k)
        = \mathcal{F}^0_\veps(w_k)
        + \mathcal{F}^1_\veps(\tilde{m}_k)
        + \mathcal{R}_k
    \end{align*}
    with the residual term
    \begin{equation*}
        \mathcal{R}_k
        = \mathcal{F}^0_\veps(m_k)
        - \mathcal{F}^0_\veps(w_k)
        = \frac{1}{2}\int_{\Omzeps}\veps^2\left(|\nabla m_k|^2-|\nabla w_k|^2 \right)\,{\dl x}.
    \end{equation*}
    Hence, to conclude the proof we only need to show that \(\mathcal{R}_k\to 0\) for \(k\to \infty\). Then (\ref{eq:liminf_superadd}) and (\ref{eq:limsup_subadd}) will follow by super- and subadditivity of \(\limsup\) and \(\liminf\). We compute with Hölder's inequality that
    \begin{align*}
        |\mathcal{R}_k|
        &\le \int_{\Omzeps}\veps^2 |\nabla m_k:\nabla \tilde{m}_k|\,{\dl x}
        +\frac{1}{2}\int_{\Omzeps}\veps^2|\nabla \tilde{m}_k|^2 \,{\dl x}\\
        &\le \left(\veps \norm{\veps \nabla m_k}{L^2(\Om;\R^{3\times 3})}+\frac{1}{2}\veps ^2\right)
        \norm{ \nabla \tilde{m}_k}{L^2(\Om;\R^{3\times 3})}.
    \end{align*}
    This finishes the proof, because \(\norm{\veps \nabla m_k}{L^2}\) and \(\norm{ \nabla \tilde{m}_k}{L^2}\) are uniformly bounded following the assumption that \(\sequence{m}{k}\) converges in the high-contrast sense to \((\tilde{m},w)\), and by the bounds of the extension operator from Proposition~\ref{thm:target_ext}. 
\end{proof}

By using the special extension operator in Proposition \ref{thm:target_ext} that preserves the saturation constraint of the magnetizations, we narrowed down the analysis of the energies \(\sequence{\mathcal{F}^1}{\veps}\) on the matrix part to the case where again the magnetizations lie in \(W^{1,2}(\Om;\sph^2)\). Actually, we can obtain the limit behaviour as a special version of a more general homogenization result. Define the functional 
\begin{equation}\label{eq:stiff_energy_func}
\mathcal{I}_\veps(m)
:= \frac{1}{2}\int_{\Om} a_{\veps}(x)|\nabla m|^2\, {\dl x}
\end{equation}
for \(m\in  W^{1,2}(\Om;\sph^2)\), where we set $a_{\veps}(x)\coloneqq a(x/\veps)\coloneqq\chi_{(\veps_k E)\cap\Om}(x)$.
The asymptotic behavior of the functionals $\mathcal{I}_{\veps}$ is characterized by the following \(\Gamma\)-convergence result. Here it will be sufficient to use in the construction of approximate recovery sequences the classical projection \(\pi\) onto \(\sph^2\) as defined in \eqref{eq:projection}. 

\begin{proposition}\label{thm:stiff_hom}
     The family \((\mathcal{I}_\veps)_\veps\) defined in \eqref{eq:stiff_energy_func} \(\Gamma\)-converges with respect to the weak \(W^{1,2}\)-topology in \(W^{1,2}(\Om;\sph^2)\) to the limit energy functional 
    \begin{align}\label{eq:limit_functional}
    \begin{split}
        \mathcal{I}(u,m)
         :=\int_{\Om}f_{\hom}(m,\nabla m)\, {\dl x},
        \end{split}
    \end{align}
    where the homogenized energy density \(f_{\hom}\) satisfies the cell formula
    \begin{equation}\label{eq:g_hom}
        f_{\hom}(s,\xi)
        = 
        \inf_{\varphi \in W^{1,2}_{\per}(Q;T_s\sph^2)}
        \frac{1}{2}\int_{Q_1} |\xi+\nabla \varphi(z)|^2\, {\dl z}.
    \end{equation}
    \end{proposition}

    The proof of Proposition \ref{thm:stiff_hom} essentially follows from the analysis in \cite{AlDF15}. For convenience of the reader, we have included a proof in Appendix \ref{sec:app_4.2}.
    The following result concerning recovery sequences for the limit energy \(\mathcal{I}\) are direct consequences of the proof of Proposition~\ref{thm:stiff_hom}. We will use it in combination with \eqref{eq:third_stwos} in Theorem~\ref{thm:recov_mag}.
    
  \begin{corollary}\label{col:stiff_approx_recov}
  Let \(\tilde{m}\in  W^{1,2}(\Om;\sph^2)\) and \(\delta>0\).
  \begin{enumerate}[(i)]
      \item 
         There exist \(\psi^\delta_{m} \in C_c^\infty(\Om;C_\per^\infty(Q;\R^3))\) such that for
      \begin{equation*}
        \tilde{m}_\veps^\delta(x)
        \coloneqq \pi\left(\tilde{m}(x) + \veps \psi^\delta_{m}\left( x,\frac{x}{\veps }\right)\right)
		\end{equation*}
        there holds
        \begin{equation}\label{eq:stiff_limsup_approx}
            \limsup_{k\to \infty} \mathcal{I}_\veps(\tilde{m}_\veps^\delta)
            \le \mathcal{I}(\tilde{m}) + \delta.
        \end{equation}
        Moreover,
        \begin{equation}
            \tilde{m}_\veps^\delta(x) \weakly \tilde{m}(x)
            \qquad \text{in } W^{1,2}(\Om;\sph^2).
        \end{equation}
      \item\label{item:replace_stiff_recov} 
        More generally, inequality \eqref{eq:stiff_limsup_approx} is valid for every sequence \(\sequence{\tilde{m}}{\veps}\subset W^{1,2}(\Om;\sph^2)\) fulfilling
        \begin{equation}
            \nabla \tilde{m}_\veps(x) 
            \stwos \nabla \tilde{m}(x) + \nabla_z\psi^\delta_{m}(x,z)\nabla \pi(\tilde{m}(x))
            \quad \text{in } L^2(\Om;\R^{3\times 3}).
        \end{equation}
    \end{enumerate}
    \end{corollary}

\subsection{Lower bound}\label{subsec:lower}
The liminf inequality in Theorem \ref{thm:main_gamma_conv} is now a straight-forward result of the definition of high-contrast convergence in Definition~\ref{def:conv_for_gamma} and the two-scale behaviour of the demagnetizing fields in Lemma~\ref{lem:demag_limit}.
\begin{proposition}
    Let \(\sequence{\veps}{k}\) be an infinitesimal sequence. For any \((\tilde{m},w)\in W^{1,2}(\Om;\R^3)\times L^2(\Om;W_0^{1,2}(Q_0;\R^3))\) and for any sequence \(\sequence{m}{k}\subset W^{1,2}(\Om;\sph^2)\) that converges to \((\tilde{m},w)\) in the strong high-contrast sense, there holds
            \begin{equation*}
                \mathcal{G}(\tilde{m},w)
                \le \liminf_{k \to \infty} \mathcal{G}_\veps(m_k).
            \end{equation*}
\end{proposition}
\begin{proof}
 We may assume that \(\liminf_{k \to \infty} \mathcal{G}_k(m_k)\) is finite, for otherwise, there is nothing to show. 
 By Definition \ref{def:conv_for_gamma} we know that there exists a sequence \(\sequence{\tilde{m}}{k}\subset W^{1,2}(\Om;\sph^2)\) such that \(\tilde{m}_k=m_k\) almost everywhere on \(\Om_k^1\), and with \(\tilde{m}_k\weakly \tilde{m}\) weakly in \(W^{1,2}(\Om;\sph^2)\). Furthermore,  \(m_k\stwos \tilde{m}+w\) in \(L^2(\Om;\R^3)\). Employing the compactness Lemma~\ref{lem:ts_key_lemma}, and from the uniqueness of the two-scale limit, we even know that there exists a subsequence -- which we do not relabel -- such that 
\begin{equation}
    \veps_k \nabla m_k(x)
    \wtwos \nabla _z w(x,z)
    \qquad \text{in } L^2(\Om;\R^{3\times 3}).
\end{equation}
Setting \(w_k:=m_k-\tilde{m}_k\), by Proposition \ref{prop:splitting} we infer that 
$$\liminf_{k\to +\infty}
        \mathcal{F}_\veps(m_k)
        \geq \liminf_{k\to +\infty}
        \mathcal{F}^0_k(,w_k) + 
        \liminf_{k\to +\infty}
\mathcal{F}^1_k(\tilde{m}_k).$$
By classical properties of two-scale convergence (see Lemma~\ref{lem:ts_intermediate} \eqref{item:ts_int_weak}) we find immediately that
 \begin{align}
 \begin{split}
     \liminf_{k\to \infty}\mathcal{F}^0_k(w_k)
     = \liminf_{k\to \infty} \frac{1}{2}\int_{\Om_k^0} \veps_k ^2|\nabla m|^2\, {\dl x}
        &\geq  \frac{1}{2}\int_\Om\int_{Q_0}|\nabla_z w(x,z)|^2\,{\dl z}{\dl x}
        = \mathcal{F}^0(w).
 \end{split}
 \end{align}
 To establish a lower bound for the matrix-related part of the energy, we apply instead Proposition \ref{thm:stiff_hom} with a slight modification due to the fact that \(\chi_k^1\) is not periodic on \(\Om\). We first conclude with Proposition \ref{thm:stiff_hom} that
 \begin{equation}
     \liminf_{k\to \infty} 
     \frac{1}{2}\int_{(\veps_k E)\cap\Om} |\nabla \tilde{m}_k|^2\,{\dl x}
     \ge  \mathcal{F}^1(\tilde{m}).    
 \end{equation}
 Then, since \((\veps_k E)\cap\Om\subset \Om_k^1\), there holds
\begin{align}
     \begin{split}
    \liminf_{k\to \infty} \mathcal{F}^1_\veps(\tilde{m}_k)
    = \liminf_{k\to \infty}\frac{1}{2}\int_{\Om_k^1}|\nabla \tilde{m}_k|^2 \,{\dl x}
    \ge  \mathcal{F}^1(\tilde{m}).
     \end{split}
 \end{align}
 To summarize, there holds
 \begin{equation}
 \label{eq:liminf1}
 	\liminf_{k\to \infty}\mathcal{F}_k(m_k)
 	\ge \mathcal{F}^0(w) + \mathcal{F}^1(\tilde{m}),
 \end{equation}
 where \(\mathcal{F}_k\) is defined in \eqref{eq:minus_self} as the difference between \(\mathcal{G}_k\) and the magnetostatic self-energy.
 In view of the superadditivity of the limes inferior, it remains to prove that 
 \begin{equation*}
 	\liminf_{k \to \infty}\int_{\R^3} |h_d[m_k]|^2\,{\dl x}
 	\ge \int_{\R^3}|h_d[(\tilde{m}+\mv{w(x,\cdot)}{Q_0})]|^2\,{\dl x} 
    + \int_{\Om}\int_{Q_0} |h_d^z[w(x,z)]|^2\,{\dl z}{\dl x}.
 \end{equation*} 
 This latter inequality is a consequence of the weak two-scale behaviour of the demagnetizing fields and of a well-known orthogonality result in two-scale convergence. From Lemma~\ref{lem:demag_limit}, by the weak two-scale convergence of \(\sequence{m}{k}\) to \(\tilde m+w\), then the corresponding demagnetizing fields \((h_d[m_k])_k\) weakly two-scale converge, up to a subsequence, in \(L^2(\R^3;\R^3)\) to 
 \begin{equation*}
 	h_d(x,z)
 	\coloneqq h_d[(\tilde{m}(x)+\mv{w(x,\cdot)}{Q_0})](x) + h_d^z[w](x,z)\chi_\Om(x).
 \end{equation*}
 Note that we used that \(h_d^z[m]=h_d^z[w]\) in our case.
 Thus, again by lower semi-continuity (see Lemma \ref{lem:ts_intermediate} \eqref{item:ts_int_weak}) we have 
 \begin{equation}
 \label{eq:liminf2}
 	\liminf_{k \to \infty}\int_{\R^3} |h_d[m_k]|^2\,{\dl x}
 	\ge \int_{\R^3}\int_Q |h_d[(\tilde{m}(x)+\mv{w(x,\cdot)}{Q_0})]
 	+ h_d^z[w]\chi_\Om(x)|^2\,{\dl z}{\dl x}.
 \end{equation}
 We recall that the function \(h_d^z[w]\) is given through the identity \(h_d^z[m]=\nabla_z r_m\) with a uniquely defined scalar potential \(r_w\in L^2(\Om;W^{1,2}_\per(Q;\R)/\R)\). In particular, \(r(x,\cdot)\) is periodic with respect to the \(z\)-variable, entailing that the \(z\)-gradient has vanishing mean value. Thus, we can orthogonally decompose the right-hand side above as 
  \begin{align}
  	\nonumber\int_{\R^3}&\int_Q |h_d[(\tilde{m}(x)+\mv{w(x,\cdot)}{Q_0})]
 	+ \nabla_z r_w|^2\,{\dl z}{\dl x}\\
  	&\label{eq:liminf3}= \int_{\R^3}|h_d[(\tilde{m}(x)+\mv{w(x,\cdot)}{Q_0})]|^2{\dl x} 
  	+ \int_{\R^3}\int_{Q_0}|h_d^z[w]|^2\chi_\Om(x)\,{\dl z}{\dl x}.
  \end{align}
The claimed lower bound follows then by combining \eqref{eq:liminf1}--\eqref{eq:liminf3}.    
\end{proof}

\subsection{Upper bound}\label{subsec:upper}
The main task in showing the limsup inequality in Theorem \ref{thm:main_gamma_conv} is the construction of a suitable approximate recovery sequence that allows to pass to the limit in all the different energy contributions. It is here that Theorem~\ref{thm:strong_ts_approx} (to be precise, its augmented version in Theorem~\ref{thm:recov_mag}) becomes essential. For a given pair \((\tilde{m},w)\), we construct the recovery sequence by perturbing the base magnetization \(\tilde{m}\) by damped oscillations to account for the energy on the matrix component; we then add the field \(w\) on the inclusions to match the high-contrast scaling of the functional there. Eventually, we pass to the limit with the help of all strong two-scale convergence properties in \eqref{eq:first_stwos}-\eqref{eq:third_stwos} of Theorem~\ref{thm:recov_mag}.
\begin{proposition}
For any \((\tilde{m},w)\in W^{1,2}(\Om;\R^3)\times L^2(\Om;W_0^{1,2}(Q_0;\R^3))\) there exists a sequence \(\sequence{m}{\veps}\subset W^{1,2}(\Om;\sph^2)\) that converges to \((\tilde{m},w)\) in the strong high-contrast sense and that satisfies
            \begin{equation*}
                \limsup_{\veps \to 0} \mathcal{G}_\veps(m_\veps)
                \le \mathcal{G}(\tilde{m},w).
            \end{equation*}
\end{proposition}
\begin{proof}
Fix $\delta>0$. 
Corollary \ref{col:stiff_approx_recov} entails that there exist \(\psi^\delta_{m} \in C_c^\infty(\Om;C_\per^\infty(Q;\R^3))\) such that
\begin{equation*}
    \limsup_{k\to \infty} \mathcal{F}_k\left(\pi\left(\tilde{m}(x) + \veps \psi^\delta_{m}\left( x,\frac{x}{\veps }\right)\right)\right)
    \le \mathcal{F}^1(\tilde{m}) + \delta.
\end{equation*}
Let now further $(m_k)_k$ be the sequence of maps provided by Theorem \ref{thm:recov_mag}, with the special choice \(\psi=\psi^\delta_{m}\). We first check the strong high contrast convergence of the sequence (see Definition \ref{def:conv_for_gamma}). First, from \eqref{eq:first_stwos} in Theorem \ref{thm:recov_mag} we infer that
\begin{equation*}
    m_k\stwos \tilde{m}+w
	\quad \text{strongly two-scale in } L^2(\Om;\R^3).
\end{equation*}
We now prove that
\begin{equation*}
    m_k\to \tilde{m}
    \quad \text{in the sense of extensions}.
\end{equation*}
Indeed, take an arbitrary subsequence of \((m_k)_k\), say \((m_{k(j)})_j\), and extend the restrictions \(m_{k(j)}\chi^1_{k(j)}\) to $\Omega_k^0$ with the estension operator provided by Proposition \ref{thm:target_ext}. Denote by \(\tilde{m}_{k(j)}\) these extensions. From the boundedness of \((m_k)_k\) and the boundedness of the extension operator, we extract a further subsequence (not relabeled) that converges weakly in \(W^{1,2}(\Om;\R^3)\) and strongly in \(L^2(\Om;\R^3)\) to a limit function \(\tilde{m}'\in W^{1,2}(\Om;\sph^2)\). At the same time, because strong \(L^2\)-convergence implies weak two-scale convergence (compare Lemma \ref{lem:ts_intermediate} \eqref{item:ts_int_weak}), we find that \(m_{k(j)}\chi^1_{k(j)}=\tilde{m}_{k(j)}\chi^1_{k(j)}\) weakly two-scale converges to \(\chi^1(z)\tilde{m}'(x)=\chi^1(z)\tilde{m}(x)\) in \(L^2(\Om;\R^3)\). This yields that \(\tilde{m}'=\tilde{m}\). Since this result is independent of the subsequence that we chose in the beginning, we conclude that \((m_k)_k\) converges to \(\tilde{m}\) in the sense of extensions. 

Coming back to the \(\limsup\)-inequality, we define 
\begin{equation*}
    \tilde{m}_k
    \coloneqq \sfT_k \left(m_k|_{\Om_k^1}\right)
    \quad \text{and} \quad 
    w_k
    \coloneqq m_k- \tilde{m}_k, 
\end{equation*}
where \(\sfT_k:	W^{1,p}(\Om_k^1;\sph^2) \to W^{1,p}(\Om;\sph^2)\) is the extension operator from Proposition \ref{thm:target_ext}.
Note that \(\sequence{w}{k}\subset W^{1,2}_0(\Om_k^0;\R^3)\). Then, it follows from the splitting in Proposition \ref{prop:splitting} that 
\begin{equation*}
    \limsup_{k\to \infty}\mathcal{G}_k (m_k)
    \le \limsup_{k\to +\infty}
        \mathcal{F}^0_k(w_k) +
        \limsup_{k\to +\infty}
        \mathcal{F}^1_k(\tilde{m}_k)
        + \limsup_{k\to +\infty}\mathcal{W}(m_k).
\end{equation*}
From \eqref{eq:second_stwos} of Theorem \ref{thm:recov_mag} we obtain 
\begin{equation*}
    \limsup_{k\to +\infty}\mathcal{F}^0_k(w_k)
    = \frac{1}{2}\int_{\Om}\int _{Q_0} |\nabla_z w(x,z)|^2\,{\dl z}{\dl x}.
\end{equation*}
For the magnetostatic self-energy $\mathcal{W}$ we use again the fact that by Lemma~\ref{lem:demag_limit} the sequence \((h_d[m_k])_k\) converges weakly two-scale in \(L^2(\R^3;\R^3)\) to
 \begin{equation*}
      h_d(x,z)
      \coloneqq h_d[(\tilde{m}(x)+\mv{w(x,\cdot)}{Q_0})] + h_d^z[w]\chi_\Om(x).
 \end{equation*}
 By Theorem \ref{thm:recov_mag}, equation \eqref{eq:first_stwos}, we thus obtain
 \begin{align*}
     \lim_{k\to +\infty}\int_{\R^3}\mathcal{W}(m_k)
     &= \lim_{k\to +\infty} - \int_{\Om } h_d[m_k]\cdot m_k\,{\dl x}\\
     &= - \int_{\Om }\int_{Q} h_d(x,z)\cdot (\tilde{m}(x) + w(x,z)\chi^0(z))\,{\dl z}{\dl x}
 \end{align*}
 and a direct computation shows that 
 \begin{align*}
     &- \int_{\Om }\int_{Q} h_d(x,z)\cdot (\tilde{m}(x) + w(x,z)\chi^0(z))\,{\dl z}{\dl x}\\
     &\quad = \int_{\R^3}|h_d[(\tilde{m}+\mv{w(x,\cdot)}{Q_0})]|^2\,{\dl x} 
    + \int_{\R^3}\int_{Q_0}|h_d^z[w]|^2\chi_\Om(x)\,{\dl z}{\dl x}.
 \end{align*}
Hence, by \cite[Section 1.2]{Br02} it remains only to prove that
\begin{equation}\label{eq:stiff_limsup_claim}
        \limsup_{k\to \infty} \mathcal{F}^1_k(\tilde{u}_k,\tilde{m}_k)
        \le \mathcal{F}^1(\tilde{u},\tilde{m}) +\delta.
    \end{equation}
As already observed in the proof of the lower bound, we have to be careful in applying Proposition \ref{thm:stiff_hom} because the characteristic function $\chi_k^1$ is not periodic. However, for
\begin{align*}\label{eq:approx_equal}
    \begin{split}
        \mathcal{R}_k
        \coloneqq \frac{1}{2}\int_{\Om_k^1}|\nabla \tilde{m}|^2 \,{\dl x}
        -\int_{(\veps_k E)\cap\Om} |\nabla \tilde{m}|^2 \,{\dl x}
        = \int_{\Om_k^1\setminus \ol{(\veps_k E)\cap \Om}} |\nabla \tilde{m}|^2\,{\dl x} 
    \end{split}
\end{align*}
we see from the $2$-equiintegrability of $(\nabla \tilde{m}_k)_k$ (with \eqref{eq:third_stwos}), as well as from the fact that \(|\Om_k^1\setminus \ol{(\veps_k E)\cap \Om}|\to 0\) as $k\to +\infty$, that $|\mathcal{R}_k|\to 0$ as $k\to +\infty$.
Finally, from \eqref{eq:third_stwos} and Corollary \ref{col:stiff_approx_recov} \eqref{item:replace_stiff_recov}, we infer that
\begin{align*}
    \begin{split}
        \limsup_{k\to \infty} \mathcal{F}^1_k(\tilde{m}_k)
        = \limsup_{k\to \infty}\int_{(\veps_k E)\cap\Om} |\nabla \tilde{m}|^2 \,{\dl x}
        \le \mathcal{F}^1(\tilde{u},\tilde{m}) +\delta.
    \end{split}
\end{align*}
This yields \eqref{eq:stiff_limsup_claim} and concludes the proof of Theorem \ref{thm:main_gamma_conv}. 
\end{proof}

	
\appendix
\section{High-contrast compactness}\label{sec:app_ts}
We provide here a proof of Lemma~\ref{lem:ts_key_lemma}.
\begin{proof}[Proof of Lemma~\ref{lem:ts_key_lemma}]
    We will deliberately switch to subsequences below without relabeling the sequence itself.
    The assumptions together with inequality (\ref{eq:key_est}) entail that \(\sequence{m}{k}\) and \(\sequence{\veps_k\nabla m}{k}\) are bounded in \(L^2(\Om;\R^3)\). Thus, usual two-scale compactness properties (cf. \cite[Proposition~1.14]{Al92}) yield the existence of \(m\in L^2(\Om;W^{1,2}_\per(Q; \R^3))\) such that, up to subsequences,
    \begin{equation}\label{eq:weak_2s_limits}
        m_k \wtwos m,
	\quad \veps _k\nabla m_k \wtwos \nabla_z m
	\qquad \text{ weakly two-scale in } L^2(\Om;\R^3).
    \end{equation}
    Additionally, there exists \(m_1\in L^2(\Om;\R^3)\) such that, up to a subsequence,
    \begin{equation}
    	\chi_k^1m_k
   	\weakly m_1
    	\qquad 
    	\text{ weakly in } L^2(\Om;\R^3).
    \end{equation}
    Now, from \eqref{eq:stwos_char} and Proposition \ref{prop:wts_unfold} we infer that on the one hand
    \begin{equation*}
        \veps_k\chi_k^1\nabla m_k
        \wtwos \chi^1\nabla_z m
        \qquad \text{ weakly two-scale in } L^2(\Om;\R^{3\times 3}).
    \end{equation*}
    On the other hand, since \(\sequence{\chi_{k}^1\nabla m}{k}\) is bounded in \(L^2(\Om;\R)\) by (\ref{eq:key_est}), we also have $$\veps_k\chi_k^1\nabla m_k\to 0 \quad\text{strongly in }L^2(\Om;\R^n).$$ Consequently, we find that \(\nabla_z m(x,z)=0\) on \(\Om\times Q_1\), and in particular that \(m\) is independent of the micro-variable \(z\) in \(Q_1\), i.e., that with Lemma \ref{lem:ts_intermediate}\eqref{item:ts_int_weak} yields
    \begin{equation}\label{eq:weak_2s_limit_stiff}
    	\chi^1(z)m(x,z)
    	= \frac{1}{|Q_1|} \chi^1(z)m_1(x)
   	\qquad \text{ for a.e. } x\in \Om.
    \end{equation} 
    Next, we will show that, up to a subsequence, there holds
    \begin{equation}\label{eq:extension_lim}
    \tilde{m}_k \weakly \frac{m_1}{|Q_1|}
    \qquad \text{ weakly in } W^{1,2}(\Om;\R^3).
    \end{equation}
    This implies that \(m_1\) is actually an element of \(W^{1,2}(\Om;\R^3)\). By virtue of estimate (\ref{eq:key_est}) and the properties of the extensions, we conclude that there exists \(\tilde{m}\in W^{1,2}(\Om;\R^3)\) and a subsequence of \(\sequence{\tilde{m}}{k}\) that converges weakly in \(W^{1,2}(\Om;\R^3)\) and even strongly in \(L^2(\Om;\R^3)\) to the limit \(\tilde{m}\). Because strong \(L^2\)-convergence implies weak two-scale convergence as stated in Lemma \ref{lem:ts_intermediate} \eqref{item:ts_int_weak}, it follows readily that 
    \begin{equation*}
    \chi_k^1m_k=\chi_k^1\tilde{m}_k
    \wtwos \chi^1 \tilde{m}
    \qquad \text{ weakly two-scale in } L^2(\Om;\R^3),
   \end{equation*}
    and then by \eqref{eq:weak_2s_limits}, \eqref{eq:weak_2s_limit_stiff}, and the uniqueness of the two-scale limit we find 
    \begin{equation*}
        \tilde{m}= \frac{m_1}{|Q_1|}.
    \end{equation*}
    Defining
    \begin{equation}
        w(x,z):= m(x,z)-\tilde{m}(x),
    \end{equation}
    we have that \(w\in L^2(\Om;W^{1,2}_0(Q_0;\R^3))\) and we obtain \eqref{eq:decomoposition}. Properties \eqref{eq:elements}, \eqref{eq:ts_convergences}, and \eqref{eq:weak_conv_ext} follow by construction. This concludes the proof.
\end{proof}
\section{\(\Gamma\)-limit of the matrix related energy}
\label{sec:app_4.2}
We present a concise proof of the \(\Gamma\)-limit claimed in Proposition~\ref{thm:stiff_hom}. To that end we fix \(m\in  W^{1,2}(\Om;\sph^2)\). 
        
\noindent{\it Step 1: Lower bound.}
Let \((m_\veps)_\veps\) be a sequence converging in \(W^{1,2}(\Om;\sph^2)\) weakly to \(m\). We first show that
\begin{equation}\label{eq:liminf}
    \mathcal{I}(m)
    \le \liminf_{\veps\to 0}\mathcal{I}_\veps(m_\veps).
\end{equation}
Since the sequence \((m_\veps)_\veps\) is uniformly bounded in \(W^{1,2}(\Om;\sph^2)\), we know by virtue of classical two-scale compactness properties (cf. \cite[Proposition~1.14]{Al92}) that there exists  \(m_1\in L^2(\Om;W^{1,p}_\per(Q;\R^n)/\R)\) such that, up to a subsequence, \(\sequence{\nabla m}{\veps}\) weakly two-scale converges to \(\nabla m(x) + \nabla_z m_1(x,z)\). Moreover, by \cite[Proposition~3.2]{AlDF15} there holds
\begin{equation*}
    m_1(x,z)
    \in T_{m(x)}\sph^2
    \text{ for a.e. } (x,z)\in \Om\times Q.
\end{equation*}
Using additionally that 
\begin{equation*}
    S_\veps(a_{\veps})(x,z)=S_\veps(\chi_{(\veps_k E)\cap\Om})(x,z)
    =\chi_1(z)
\end{equation*}
we conclude from Lemma~\ref{lem:ts_intermediate} \eqref{item:ts_int_weak} immediately that
\begin{align*}
    \liminf_{\veps\to 0} \frac{1}{2}\int_{\Omega}a_{\veps}|\nabla m_{\veps}|^2\, {\dl x}
    &= \liminf_{\veps\to 0} \frac{1}{2}\int_{\Omega}\int_{Q_1}|S_\veps(\nabla m_{\veps})|^2\, {\dl z}{\dl x}\\
    &\ge \frac{1}{2}\int_{\Omega}\int_{Q_1}|\nabla m(x) + \nabla_z m_1(x,z)|^2\, {\dl z}{\dl x}\\
    &\ge \int_\Om\Bigg[\inf_{\varphi \in W^{1,2}_{\per}(Q;T_{m(x)}\sph^2)}
    \frac{1}{2}\int_{Q_1} |\nabla m+\nabla \varphi(z)|^2\, {\dl z}\Bigg]{\dl x}.
\end{align*}

\noindent{\it Step 2: Upper bound.}
For \(\psi_m\in W^{1,2}(\Om; C^\infty_{\per}(Q;\R^3))\), we define the perturbation
\begin{equation}\label{eq:perturbations}
    m_\veps(x)\coloneqq \pi\left(m(x) + \veps \psi_m\left( x,\frac{x}{\veps }\right)\right),
\end{equation}
where \(\pi\) is the customary projection onto \(\sph^2\) (compare \eqref{eq:projection}). The functions \(m_\veps\) are in particular well-defined for \(\veps\) sufficiently small. There holds
\begin{equation}\label{eq:nec_l2_conv}
    m_\veps \to m
    \qquad \text{ strongly in } L^2(\Om;\R^3).
\end{equation} 
Furthermore, owing to the argument in \cite[Theorem 2.3, line (60)]{DaDF20} there also holds 
\begin{equation}\label{eq:stwos_pert}
    \nabla m_\veps
    \stwos \nabla m + \nabla_z\psi_m \nabla \pi(m)
    \qquad \text{ strongly in } L^2(\Om;\R^{3\times 3}).
\end{equation}
We claim now that for every \(\delta>0\) we can chose \(\psi_m=\psi_m^\delta\) such that the corresponding sequence \((m_\veps^\delta)_\veps\) defined in \eqref{eq:perturbations} fulfils
\begin{equation}
    \limsup_{\veps\to 0} \mathcal{I}_\veps(u_\veps^\delta, m_\veps^\delta)
    \le \mathcal{I}(u,m) + \delta.
\end{equation}
By classical $\Gamma$-convergence arguments (cf. \cite[Section 1.2]{Br02}), this would then establish the \(\Gamma-\limsup\) inequality. We first observe that using the strong two-scale convergence in \eqref{eq:stwos_pert} we have, similarly to step~1, 
\begin{equation}\label{eq:limsup_aux}
    \lim_{\veps \to 0}  \frac{1}{2}\int_{\Omega}a_{\veps}|\nabla m_{\veps}|^2\, {\dl x}
    = \frac{1}{2}\int_{\Omega}a_{\veps}|\nabla m + \nabla_z\psi_m \nabla \pi(m)|^2\, {\dl x}.
\end{equation}
Now, because \(C^\infty_{\per}(Q;\R^3)\) is dense in \(W^{1,2}_{\per}(Q;\R^3)\) and the right-hand side of \eqref{eq:limsup_aux} is continuous with respect to the \(W^{1,2}(Q;\R^3)\)-topology, we conclude that there exists \(\psi_m^\delta\in W^{1,2}(\Om; C^\infty_{\per}(Q;\R^3))\) such that
\begin{equation*}
    \lim_{\veps \to 0}  \frac{1}{2}\int_{\Omega}a_{\veps}|\nabla m_{\veps}|^2\, {\dl x}
    \le \int_\Om\Bigg[\inf_{\varphi \in W^{1,2}_{\per}(Q;T_{m(x)}\sph^2)}
    \frac{1}{2}\int_{Q_1} |\nabla m+\nabla \varphi(z)|^2\, {\dl z}\Bigg]{\dl x} +\delta,
\end{equation*}
where we used additionally that \(\nabla \varphi\nabla \pi(m)=\nabla \varphi\) for all \(\varphi \in W^{1,2}_{\per}(Q;T_m\sph^2)\).
This concludes the proof.


\section*{Acknowledgements}
E.D. acknowledges support
from the Austrian Science Fund (FWF) projects
\href{https://doi.org/10.55776/F65}{10.55776/F65},  \href{https://doi.org/10.55776/Y1292}{10.55776/Y1292}, and \href{https://doi.org/10.55776/P35359}{10.55776/P35359}, as well as
from the OeAD-WTZ project CZ04/2019 (M\v{S}MT\v{C}R 8J19AT013). The research of L.H. has been funded from the Austrian Science Fund (FWF) through project  \href{https://doi.org/10.55776/P34609}{10.55776/P34609}. L.H. also thanks TU Wien and MedUni Wien for their support and hospitality. Both authors are thankful to C. Gavioli and V. Pagliari for some interesting discussions on the topic of high-contrast materials.


\bibliography{master.bib}

\begin{thebibliography}{10}

\bibitem{AcPiMaPe92}
E.~Acerbi, V.~Chiad\`o~Piat, G.~Dal~Maso, and D.~Percivale.
\newblock An extension theorem from connected sets, and homogenization in
  general periodic domains.
\newblock {\em Nonlinear Anal.}, 18(5):481--496, 1992.
\newblock \href {https://doi.org/10.1016/0362-546X(92)90015-7}
  {\path{doi:10.1016/0362-546X(92)90015-7}}.

\bibitem{Al92}
G.~Allaire.
\newblock Homogenization and two-scale convergence.
\newblock {\em SIAM J. Math. Anal.}, 23(6):1482--1518, 1992.
\newblock \href {https://doi.org/10.1137/0523084} {\path{doi:10.1137/0523084}}.

\bibitem{AlDF15}
F.~Alouges and G.~Di~Fratta.
\newblock Homogenization of composite ferromagnetic materials.
\newblock {\em Proc. A.}, 471(2182):20150365, 19, 2015.
\newblock \href {https://doi.org/10.1098/rspa.2015.0365}
  {\path{doi:10.1098/rspa.2015.0365}}.

\bibitem{BaMi10}
J.-F.~c. Babadjian and V.~Millot.
\newblock Homogenization of variational problems in manifold valued {S}obolev
  spaces.
\newblock {\em ESAIM Control Optim. Calc. Var.}, 16(4):833--855, 2010.
\newblock \href {https://doi.org/10.1051/cocv/2009025}
  {\path{doi:10.1051/cocv/2009025}}.

\bibitem{BeZh88}
F.~Bethuel and X.~M. Zheng.
\newblock Density of smooth functions between two manifolds in {S}obolev
  spaces.
\newblock {\em J. Funct. Anal.}, 80(1):60--75, 1988.
\newblock \href {https://doi.org/10.1016/0022-1236(88)90065-1}
  {\path{doi:10.1016/0022-1236(88)90065-1}}.

\bibitem{BPV14}
P.~Bousquet, A.~C. Ponce, and J.~Van~Schaftingen.
\newblock Strong approximation of fractional {S}obolev maps.
\newblock {\em J. Fixed Point Theory Appl.}, 15(1):133--153, 2014.
\newblock \href {https://doi.org/10.1007/s11784-014-0172-5}
  {\path{doi:10.1007/s11784-014-0172-5}}.

\bibitem{Br02}
A.~Braides.
\newblock {\em {$\Gamma$}-convergence for beginners}, volume~22 of {\em Oxford
  Lecture Series in Mathematics and its Applications}.
\newblock Oxford University Press, Oxford, 2002.
\newblock \href {https://doi.org/10.1093/acprof:oso/9780198507840.001.0001}
  {\path{doi:10.1093/acprof:oso/9780198507840.001.0001}}.

\bibitem{Br63}
W.~F. Brown.
\newblock {\em Micromagnetics}.
\newblock Interscience Publishers, London, 1963.

\bibitem{ChCh12}
M.~Cherdantsev and K.~D. Cherednichenko.
\newblock Two-scale {$\Gamma$}-convergence of integral functionals and its
  application to homogenisation of nonlinear high-contrast periodic composites.
\newblock {\em Arch. Ration. Mech. Anal.}, 204(2):445--478, 2012.
\newblock \href {https://doi.org/10.1007/s00205-011-0481-4}
  {\path{doi:10.1007/s00205-011-0481-4}}.

\bibitem{CDG02}
D.~Cioranescu, A.~Damlamian, and G.~Griso.
\newblock Periodic unfolding and homogenization.
\newblock {\em C. R. Math. Acad. Sci. Paris}, 335(1):99--104, 2002.
\newblock \href {https://doi.org/10.1016/S1631-073X(02)02429-9}
  {\path{doi:10.1016/S1631-073X(02)02429-9}}.

\bibitem{CioDo99}
D.~Cioranescu and P.~Donato.
\newblock {\em An Introduction to Homogenization}, volume~17 of {\em Oxford
  Lecture Series in Mathematics and its Applications}.
\newblock The Clarendon Press, Oxford University Press, New York, 1999.

\bibitem{DM93}
G.~Dal~Maso.
\newblock {\em An introduction to {$\Gamma$}-convergence}, volume~8 of {\em
  Progress in Nonlinear Differential Equations and their Applications}.
\newblock Birkh\"{a}user Boston, Inc., Boston, MA, 1993.
\newblock \href {https://doi.org/10.1007/978-1-4612-0327-8}
  {\path{doi:10.1007/978-1-4612-0327-8}}.

\bibitem{DaDF20}
E.~Davoli and G.~Di~Fratta.
\newblock Homogenization of chiral magnetic materials: a mathematical evidence
  of {D}zyaloshinskii's predictions on helical structures.
\newblock {\em J. Nonlinear Sci.}, 30(3):1229--1262, 2020.
\newblock \href {https://doi.org/10.1007/s00332-019-09606-8}
  {\path{doi:10.1007/s00332-019-09606-8}}.

\bibitem{DaGaPa2}
E.~Davoli, C.~Gavioli, and V.~Pagliari.
\newblock Homogenization of high-contrast media in finite-strain
  elastoplasticity.
\newblock {\em ArXiv preprint}, 2023.
\newblock URL: \url{https://arxiv.org/abs/2301.02170}.

\bibitem{DaKrPa21}
E.~Davoli, M.~Kru\v{z}\'ik, and V.~Pagliari.
\newblock Homogenization of high-contrast composites under differential
  constraints.
\newblock {\em Advances in Calculus of Variations}, 17(2):277--318, 2024.
\newblock URL: \url{https://doi.org/10.1515/acv-2022-0009}, \href
  {https://doi.org/doi:10.1515/acv-2022-0009}
  {\path{doi:doi:10.1515/acv-2022-0009}}.

\bibitem{FoKr10}
I.~Fonseca and S.~Krömer.
\newblock Multiple integrals under differential constraints: two-scale
  convergence and homogenization.
\newblock {\em Indiana Univ. Math. J.}, 59(2):427--457, 2010.

\bibitem{GaHaPa23}
C.~Gavioli, L.~Happ, and V.~Pagliari.
\newblock An extension operator for manifold-valued sobolev maps on perforated
  domains, 2024.
\newblock URL: \url{https://arxiv.org/abs/2403.11690v2}, \href
  {https://arxiv.org/abs/2403.11690} {\path{arXiv:2403.11690}}.

\bibitem{Ha09}
P.~Haj{\l}asz.
\newblock Sobolev mappings between manifolds and metric spaces.
\newblock In {\em Sobolev Spaces In Mathematics I: Sobolev Type Inequalities},
  volume~8 of {\em Int. Math. Ser. (N.\ Y.)}, pages 185--222. Springer, New
  York, 2009.
\newblock \href {https://doi.org/10.1007/978-0-387-85648-3\_7}
  {\path{doi:10.1007/978-0-387-85648-3\_7}}.

\bibitem{HuSch08}
A.~Hubert and R.~Schäfer.
\newblock {\em Magnetic Domains: The Analysis of Magnetic Microstructures}.
\newblock Springer-Verlag, Berlin, 2008.

\bibitem{Lee03}
J.~M. Lee.
\newblock {\em Introduction to smooth manifolds}, volume 218 of {\em Graduate
  Texts in Mathematics}.
\newblock Springer-Verlag, New York, 2003.
\newblock \href {https://doi.org/10.1007/978-0-387-21752-9}
  {\path{doi:10.1007/978-0-387-21752-9}}.

\bibitem{LNW02}
D.~Lukkassen, G.~Nguetseng, and P.~Wall.
\newblock Two-scale convergence.
\newblock {\em Int. J. Pure Appl. Math.}, 2(1):35--86, 2002.

\bibitem{Ng89}
G.~Nguetseng.
\newblock A general convergence result for a functional related to the theory
  of homogenization.
\newblock {\em SIAM J. Math. Anal.}, 20(3):608--623, 1989.
\newblock \href {https://doi.org/10.1137/0520043} {\path{doi:10.1137/0520043}}.

\bibitem{Pi04}
G.~Pisante.
\newblock Homogenization of micromagnetics large bodies.
\newblock {\em ESAIM Control Optim. Calc. Var.}, 10(2):295--314, 2004.
\newblock \href {https://doi.org/10.1051/cocv:2004008}
  {\path{doi:10.1051/cocv:2004008}}.

\bibitem{SR07}
K.~Santugini-Repiquet.
\newblock Homogenization of the demagnetization field operator in periodically
  perforated domains.
\newblock {\em J. Math. Anal. Appl.}, 334(1):502--516, 2007.
\newblock \href {https://doi.org/10.1016/j.jmaa.2007.01.001}
  {\path{doi:10.1016/j.jmaa.2007.01.001}}.

\bibitem{SchUh83}
R.~Schoen and K.~Uhlenbeck.
\newblock Boundary regularity and the dirichlet problem for harmonic maps.
\newblock {\em J. Differential Geometry}, 18:253--268, 1983.

\bibitem{Att86}
R.~B. Vinter.
\newblock Variational convergence for functions and operators (applicable
  mathematics series).
\newblock {\em Bulletin of the London Mathematical Society}, 18(2):222--223,
  1986.
\newblock URL:
  \url{https://londmathsoc.onlinelibrary.wiley.com/doi/abs/10.1112/blms/18.2.222},
  \href
  {https://arxiv.org/abs/https://londmathsoc.onlinelibrary.wiley.com/doi/pdf/10.1112/blms/18.2.222}
  {\path{arXiv:https://londmathsoc.onlinelibrary.wiley.com/doi/pdf/10.1112/blms/18.2.222}},
  \href {https://doi.org/10.1112/blms/18.2.222}
  {\path{doi:10.1112/blms/18.2.222}}.

\bibitem{Vi04}
A.~Visintin.
\newblock Some properties of two-scale convergence.
\newblock {\em Atti Accad. Naz. Lincei Cl. Sci. Fis. Mat. Natur. Rend. Lincei
  (9) Mat. Appl.}, 15(2):93--107, 2004.

\bibitem{Vi06}
A.~Visintin.
\newblock Towards a two-scale calculus.
\newblock {\em ESAIM Control Optim. Calc. Var.}, 12(3):371--397, 2006.
\newblock \href {https://doi.org/10.1051/cocv:2006012}
  {\path{doi:10.1051/cocv:2006012}}.

\end{thebibliography}
\bibliographystyle{abbrvurl}

\end{document}